\documentclass[12pt]{article}

\usepackage{fullpage}
\usepackage{amssymb}
\usepackage{amsmath}
\usepackage{amsthm}
\usepackage{stmaryrd}
\usepackage{xypic}
\usepackage{setspace}
\usepackage{array}
\usepackage{rotating}
\usepackage{verbatim}
\usepackage[colorlinks=true,urlcolor=blue]{hyperref}
\usepackage{graphicx}
\graphicspath{{IncludeImages/}}
\usepackage[pagewise]{lineno}

\newtheorem{theorem}{Theorem}[section]
\newtheorem{lemma}[theorem]{Lemma}
\newtheorem{proposition}[theorem]{Proposition}
\newtheorem{corollary}[theorem]{Corollary}

\newcommand{\cref}[1]{Corollary~\textup{\ref{cor:#1}}}

\newcommand{\s}{\sigma}
\newcommand{\G}{\Gamma}

\newcommand{\Stab}{\textrm{Stab}}
\newcommand{\Sym}{\textrm{Sym}}
\newcommand{\rank}{\textrm{rank}}

\newcommand{\calP}{\mathcal P}
\newcommand{\calQ}{\mathcal Q}

\newcounter{reminder} 

\begin{document}

\title{Vertex-Faithful Regular Polyhedra}

\author{Gabe Cunningham\\
University of Massachusetts Boston\\
Boston, Massachusetts \\
and \\
Mark Mixer \\
Wentworth Institute of Technology \\
Boston, Massachusetts
}

\date{ \today }
\maketitle

\begin{abstract}

We study the abstract regular polyhedra with automorphism groups that act faithfully on their vertices, and show that each non-flat abstract regular polyhedron covers a ``vertex-faithful" polyhedron with the same number of vertices.  We then use this result and earlier work on flat polyhedra to study abstract regular polyhedra based on the size of their vertex set.  In particular, we classify all regular polyhedra where the number of vertices is prime or twice a prime. We also construct the smallest regular polyhedra with a prime squared number of vertices.


\end{abstract}

\section{Introduction}
In recent years, the study of maps on compact surfaces has seen much attention. 
Often maps are investigated by considering how their automorphism groups act on faces of a certain dimension.  For instance, maps where the automorphism group acts transitively on the edges (edge-transitive maps) have been classified by Graver, \v{S}ir\'{a}\v{n}, Tucker, and Watkins~\cite{graver1997locally, EdgeTransOrientable}.

Regular maps --- those with maximal symmetry, whose automorphism groups act transitively on flags (incident triples of vertices, edges, and facets) --- have received the most consideration.
A regular map that satisfies one more condition (called the diamond property) can be seen as an abstract regular polyhedron, and the study of abstract regular polytopes has a rich history of its own~\cite{ARP}.  
	
In this paper, we will restrict our study to regular abstract polyhedra, and will be interested in both the action of the automorphism group on the vertices, and the prime factorization of the number of vertices.   Regular maps with specific actions on their vertices have been considered before; for instance, in~\cite{Jajcay2019}, where regular maps with quasiprimitive automorphism groups on vertices are discussed, and in~\cite{LiNotFaithful}, where regular maps with automorphism groups that do not act faithfully on their vertices, edges, or faces are analyzed.

Similarly, the prime factorization of the number of vertices (or faces of another dimension) has also been utilized in previous works.  In~\cite{Breda2015}, regular maps with a prime number of faces are examined, and in~\cite{Du2004} regular maps with simple underlying graphs whose order is the product of two primes are classified. Many of the maps in these classifications do not correspond to abstract polyhedra, and we find that the classification of regular polyhedra with a given number of vertices is somewhat tamer than corresponding classifications of maps.

  When examining polyhedra where the size of the vertex set has few prime factors, it is useful to understand the smallest examples fully.  Thankfully, much is known about the small regular polyhedra, as those with up to 4000 flags have been enumerated~\cite{conder-atlas}. Using these data, we observed that there appeared to be only two regular polyhedra with $b$ vertices for each prime $b \geq 5$; an observation which inspired this whole project.

This paper is structured as follows.  In Section~\ref{sec:def}  we provide some definitions and basic results about regular maps and regular abstract polyhedra.  In particular, in Section~\ref{sec:def} we will recall various results about flat regular polyhedra, and provide some basic results about regular polyhedra where the automorphism group acts faithfully on the vertices.  In Section~\ref{sec:few} we consider such polyhedra with few vertices.  Then, utilizing these results, in Sections~\ref{sec:prime},~\ref{sec:twiceprime}, and~\ref{sec:primesq}, we consider regular polyhedra where the number of vertices is prime, twice a prime, or a prime squared, respectively.   Specifically, we classify all regular polyhedra with either a prime number of vertices or twice a prime number of vertices. We also construct the smallest regular polyhedra with a prime squared number of vertices.


\section{Definitions and Useful Results}
\label{sec:def}
		

\subsection{Abstract Polyhedra}

In what follows, we recall some definitions and results from the theory of abstract regular polytopes (see~\cite{ARP}).  From here on out, we will omit the word abstract, and simply refer to polytopes and polyhedra.  A {\em polyhedron} is a ranked poset $\calP$ with the following properties.  The elements of $\calP$ are called {\em faces}, with ranks in $\{-1, \ldots, 3\}$.  The poset $\calP$ has a unique face of rank -1, a unique face of rank 3, and the faces of ranks 0, 1, and 2 are called {\em vertices}, {\em edges}, and {\em facets} respectively.   The maximal totally ordered subsets of $\calP$ are called {\em flags}, and each flag has exactly 5 faces, one for each rank.   We say that two faces are {\em incident} if they are on the same flag, and we say that two flags are {\em adjacent} if they differ by exactly one face.  In particular, we say that $\Psi$ and $\Phi$ are $i$-adjacent if they differ exactly in a face of rank $i$.
Given any two flags $\Psi$ and $\Phi$ of $\calP$, there is a sequence of flags $\Psi = \Psi_0, \Psi_1, \ldots, \Psi_k = \Phi$, so that each flag in the sequence contains $\Psi \cap \Phi$ and any two successive flags are adjacent.  Additionally, whenever $F < G$, and $\rank(F) = \rank(G)-2$, there are exactly two faces $H$ such that $F < H < G$; this is called the diamond property.
Given a vertex $v$ of a polyhedron, the collection of faces $F$ so that $v \leq F$ is incident to $v$ is called the \emph{vertex figure at $v$}.  

The {\em automorphism group} $\Gamma(\calP)$ of a polyhedron is the group of rank-preserving automorphisms of the partially ordered set.    When the size of the automorphism group is the same as the number of flags, all the flags will be in the same orbit, and the polyhedron is called {\em regular}.     In fact, each regular polyhedron can be thought of as a regular {\em map}.  Informally, a map $\calP$ is a family of finite polygons with the following four properties.  Any two polygons of the map meet in a common edge or vertex, or do not meet at all.  Each edge of the map belongs to exactly two polygons. The set of polygons containing a given vertex form a single cycle of polygons, where adjacent polygons in the cycle share a common edge.  Finally, between any two polygons is a chain of adjacent polygons.

Much is known about the automorphism groups of regular polyhedra.  If we fix a base flag $\Psi$ of $\calP$, then the automorphism group $\Gamma(\calP)$ of a regular polyhedron $\calP$ is generated by the three involutions $\rho_i$ (with $i \in \{0,1,2\}$), where $\rho_i$ sends $\Psi$ to the adjacent flag $\Psi^i$ differing in a face of rank $i$.  Thus, for example, the generator $\rho_0$ sends the base vertex to the other vertex incident to the base edge and keeps the base face fixed.  

 Any regular polyhedron $\calP$ has a (Schl\"{a}fli) type $\{p,q\}$ where each facet is a polygon with $p$ vertices, and each vertex is incident to $q$ edges.   In this case the automorphism group $\G(\calP)$ is a smooth quotient of the string Coxeter group $[p,q]$, and is called a \emph{string C-group} of rank 3, where
 
 $$[p,q] := \langle a, b, c \mid a^2 = b^2 = c^2 = (ab)^p = (bc)^q = (ac)^2 = 1 \rangle.$$
 
Each string C-group satisfies a particular {\em intersection condition} inherited from the Coxeter group.  Namely, let $\G  = \langle \rho_0, \rho_1, \rho_2 \rangle$  be the automorphism group of a regular polyhedron, and let $\G_0 := \langle \rho_1, \rho_2 \rangle$ be the stabilizer of the base vertex, and $\G_2 := \langle \rho_0, \rho_1 \rangle$ be the stabilizer of the base facet.  Then the {\em intersection condition} implies that  $\G_0 \cap \G_2 \cong \langle \rho_1 \rangle$.  In fact, every rank 3 string C-group is the
automorphism group of a unique regular polyhedron (see Theorem. 2E11 of \cite{ARP} ).  Due to this, we can refer to the (possibly infinite) universal regular polytope of type $\{p,q\}$ which has automorphism group the Coxeter group $[p,q]$.   

Let $\calP$ and $\calQ$ be regular polyhedra.  We say that $\calP$ \emph{covers} $\calQ$  if there a surjective function $\psi$ from $\calP$ to $\calQ$ that preserves incidence, rank, and has the property that if two flags of $\calP$ are $i$-adjacent then so are their images under $\psi$.  An isomorphism from a regular polyhedron $\calP$ to a regular polyhedron $\calQ$ is a bijection that preserves incidence and rank.  If $\calP$ covers $\calQ$, and $\calP$ is not isomorphic to $\calQ$, we say that $\calP$ is a \emph{proper cover} of $\calQ$.  
Furthermore, $\calP$ covers $\calQ$ if and only if there is an epimorphism from $\G(\calP)=\langle \rho_0, \rho_1, \rho_2 \rangle $ to $\G(\calQ)=\langle r_0, r_1, r_2 \rangle$ sending each $\rho_i$ to $r_i$.

Let $\G(\calP) = \langle \rho_0, \rho_1, \rho_2 \rangle$ be the automorphism group of a regular polyhedron.  It will be useful to define the {\em abstract rotations} $\s_1$ and $\s_2$, where $\s_1 := \rho_0 \rho_1$ and $\s_2 := \rho_1 \rho_2$.    The index of $\langle \s_1, \s_2 \rangle$ in $\G(\calP)$ is at most 2. When the index is exactly 2, we say that $\calP$ is {\em orientably regular}; otherwise, when the index is 1, we say that $\calP$ is {\em non-orientably regular}.    A regular polyhedron $\calP$ is orientably regular (or simply {\em orientable}) if and only if all identity words $w$ in terms of the generators $\rho_i$ have even length.  Thus $\calP$ is {\em non-orientable} when there is a trivial word $w$ in terms of the generators $\rho_i$ which has odd length.

To each polyhedron, we may associate its dual polyhedron $\calP^\delta$, which is constructed by reversing the partial order of $\calP$.   If $\calP^\delta$ and $\calP$ are isomorphic, then we say that $\calP$ is \emph{self-dual}.  
Similarly, to each regular polyhedron $\calP$ we may associate its {\em Petrie dual} (or more briefly {\em Petrial}) $\calP^\pi$.   The Petrial has the same vertices and edges as the original polyhedron.  The facets of $\calP^\pi$ are the {\em Petrie polygons} of $\calP$, defined so that exactly two successive edges of a Petrie polygon are edges of
a facet of $\calP$. Note that the Petrie dual of a polyhedron may fail to be a polyhedron.   If $\calP$ is isomorphic to $\calP^\pi$, then $\calP$ is said to be self-Petrie.   For a regular polyhedron, the length $r$ of a Petrie polygon is the order of the element $\rho_0 \rho_1 \rho_2$ in its automorphism group.  We refer to the (universal) regular polytope of type $\{p,q\}_r$ which has an automorphism group with the same presentation as the Coxeter group $[p,q]$ plus the additional relator $(\rho_0 \rho_1 \rho_2)^r$, denoted $[p,q]_r$.

For instance the (universal) polyhedron $\{4,4\}_6$ is the regular polyhedron that has automorphism group with the following presentation

 $$ \langle a, b, c \mid a^2 = b^2 = c^2 = (ab)^4 = (bc)^4 = (ac)^2 = (abc)^6= 1 \rangle.$$

We will later utilize the fact that the polytope $\{4,4\}_{2s}$ is isomorphic to the toroidal map $\{4,4\}_{(s,s)}$ (see~\cite{coxeter-moser} Section 8.6). Furthermore, the dual of the Petrial of the toroidal map $\{4,4\}_{(s,0)}$ will play an important role, and it will be denoted $(\{4,4\}_{(s,0)})^{\pi \delta}$. Note that the Petrie polygons of $\{4,4\}_{(s,0)}$ visit each vertex at most once, and so by Lemma 7B3 of \cite{ARP}, the Petrial of $\{4, 4\}_{(s,0)}$ (and thus its dual) is in fact a polyhedron.
 
 Given a regular polyhedron $\calP$ of type $\{p,q\}$, it is easy to see that $| \Gamma(\calP) | = 4e = 2vq = 2fp$ where $e$ is the number of edges, $v$ is the number of vertices, and $f$ is the number of facets. Furthermore, $\calP$ has at least $p$ vertices, and so it has at least $2pq$ flags. 
 

	\subsection{Flat regular polyhedra} \label{sec:flat}
	
		A regular polyhedron is said to be \emph{flat} if every face is incident to every vertex.
		If $\calP$ is a regular polyhedron of type $\{p, q\}$, then it is flat if and only if it is \emph{tight},
		which means that it has the minimum possible number of flags ($2pq$). 
		
		Flat (tight) regular polyhedra were classified in \cite{tight3}. The non-orientable ones fall
		into a small number of infinite families. The orientable ones are somewhat more diverse.
		From Theorem 3.3 of \cite{tight3}, we know that every flat orientably regular polyhedron has an automorphism group of the form
		\[ \Lambda(p,q)_{i,j} = [p,q] / (\s_2^{-1} \s_1 = \s_1^i \s_2^j), \]
		with $i$ and $j$ taken modulo $p$ and $q$, respectively.
		
		If we fix the number of vertices (which is equal to $p$), then we would like to know what values
		of $q$, $i$, and $j$ actually yield a flat polyhedron of type $\{p, q\}$. Here we summarize the results
		from \cite{tight3} that allow us to do so.
		
		\begin{proposition}[\cite{tight3} Proposition 4.2] \label{prop:flat-classification}
		Suppose $p,q \geq 3$ and that $0 \leq i \leq p-1$ and $0 \leq j \leq q-1$. Then
		$\Lambda(p,q)_{i,j}$ is the automorphism group of a flat orientably regular polyhedron
		of type $\{p, q\}$ if and only if there are values $p'$ and $q'$ such that:
		\begin{enumerate}
		\item $p'$ divides $p$ and $i+1$,
		\item $q'$ divides $q$ and $j-1$,
		\item The group $\Lambda(p,q')_{i,1}$ is the automorphism group of a flat orientably regular polyhedron
		of type $\{p, q'\}$, with $\langle \s_2 \rangle$ core-free in $\Lambda(p,q')_{i,1}$, and
		\item The group $\Lambda(p',q)_{-1,j}$ is the automorphism group of a flat orientably regular polyhedron
		of type $\{p', q\}$, with $\langle \s_1 \rangle$ core-free in $\Lambda(p',q)_{-1,j}$.
		\end{enumerate}
		\end{proposition}
		
		\begin{proposition} \label{prop:flat-core-free}
		The group $\Lambda(p,q')_{i,1}$ is the automorphism group of a flat orientably regular polyhedron $\calP$
		of type $\{p, q'\}$, with $\langle \s_2 \rangle$ core-free in $\Lambda(p,q')_{i,1}$, 
		if and only if one of the following is true:
		\begin{enumerate}
		\item $q' = 2$ and $i = -1$, or
		\item $q'$ is odd, $p = 2q'$, and $i = 3$, or
		\item $q'$ is an even divisor of $p$ such that 
			\begin{itemize}
			\item $\gcd(p/q', q')$ is a power of $2$
			\item If the maximal power of $2$ that divides $p$ is $2^{\alpha}$, then the maximal power of $2$
			that divides $q'$ is either $2$, $4$, or $2^{\alpha-1}$, and it is only $4$ is $\alpha \geq 3$,
			\end{itemize}
			and $i$ satisfies a particular system of congruences. In particular, if we write $p = 2^{\alpha} p_1 p_2$,
			where $p_1$ is coprime with $q'$ and $p_2$ divides $q'$, and if $2^{\beta}$ is the maximal power of $2$
			that divides $q'$, then:
			\begin{itemize}
			\item $\frac{1-i}{2} \equiv -1$ (mod $p_2$),
			\item $\frac{1-i}{2} \equiv 1$ (mod $p_1$),
			\item If $\beta = 1$, then $\frac{1-i}{2} \equiv 1$ (mod $2^{\alpha-1}$)
			\item If $\beta = 2$, then $\frac{1-i}{2} \equiv 2^{\alpha-2}+1$ (mod $2^{\alpha-1}$)
			\item If $\beta = \alpha-1$, then $\frac{1-i}{2} \equiv 2^{\alpha-2}-1$ or $-1$ (mod $2^{\alpha-1}$)
			\end{itemize}
		\end{enumerate}
		\end{proposition}
		
		The proof of the previous proposition is contained in the contents of Section 4.1 of~\cite{tight3}. Proposition~\ref{prop:flat-core-free} also has a dual version, with the roles of $p$ and $q$ reversed and with the roles of $i$ and $-j$ reversed.  We note that for tight orientable regular polyhedra, $\langle \s_2 \rangle$ is core-free in $\Lambda(p,q')_{i,1}$ if and only if $\calP$ has no multiple edges (that is, if no pair of vertices is incident to two or more edges).

	
	\subsection{Group Actions and Permutation Representation Graphs}
	
	
	In much of this paper we will be concerned with how the automorphism group of a regular polyhedron acts on its vertices.  Here we give some of the required concepts about group actions and permutation representation graphs that we will use in our proofs.  We will follow~\cite{dixon2012} for the results on permutation groups, and~\cite{CPR} for the notation about permutation representation graphs.
	
Let $\G$ be the automorphism group of a finite regular polytope $\calP$ and let $\Omega$ be a set of faces or flags of $\calP$.  Any homomorphism $f$ of $\G$ into the symmetric group $\Sym(\Omega)$ is called a (permutation) representation of $\G$ on $\Omega$, and we will say that $\G$ acts on $\Omega$.  The kernel of $f$ is called the kernel of the representation, and the representation is called faithful when the kernel is trivial; in that case, the image of $\G$ under $f$ is isomorphic to $\G$.  

For each element $i \in \Omega$  and each element $g \in \G$, we denote the image of $i$ under $g$ by $i^g$ or $(i)g$. (The former notation is standard, but since our group elements $g$ often involve superscripts, the latter notation will be more convenient for us later.) We can extend this notion to subgroups, and denote the orbit of $i$ under $H$ as $i^H : = \{ i ^ g \mid g \in H \}$, where $H \leq \G$.  When the orbit of $i$ under $H$ is all of $\Omega$, $H$ is said to act transitively.  Additionally, we can extend this notion to subsets of $\Omega$, denoting $S^g := \{ i^g \mid i \in S\}$, where $S \subseteq \Omega$.  Finally, we say a nonempty subset $B$ of $\Omega$ is a {\em block} for $\G$ if for each $g \in \G$ either $B^g = B$ or $B^g \cap B = \emptyset$.  For each $i \in \Omega$, the set $\{i\}$ forms a trivial block for $\G$, as does $\Omega$ itself.

If $\G$ acts transitively on $\Omega$ and $B$ is a block for $\G$, then the set $\Sigma = \{ B^g \mid g \in \G \}$ is a partition of $\Omega$ and each element of the partition is a block for $\G$.    The set $\Sigma$ is called a system of blocks (or a block system) for $\G$.  The group $\G$ is said to be primitive if it has no nontrivial blocks on $\Omega$.  If there is a nontrivial partition of $\Omega$ by blocks of $\G$, then $\G$ is said to be imprimitive.  As each element of the partition will have the same size (say $k$), then $|\Omega| = km$, where there are $m$ blocks each of size $k$.  

Let $\G$ act transitively on $\Omega$. There are two common types of block system we will use:
\begin{itemize}
\item If $H$ is a normal subgroup of $\G$, then the orbits of $H$ form a block system for $\G$.
\item If $\Stab(\G,i)$ is the stabilizer of some point $i \in \Omega$, then the set of fixed points of $\Stab(\G,i)$ form a block for $\G$.
\end{itemize}
We will also need Burnside's Theorem for transitive groups of prime degree:  
\begin{theorem}[Burnside, \cite{burnside}] \label{prop:group2}
A transitive permutation group of prime degree $b$ is doubly transitive or has a normal Sylow $b$-subgroup.
\end{theorem}
	
Given a regular polyhedron $\calP$, and thus a string C-group $\G(\calP)$, when $\G(\calP)$ acts faithfully on a set $\Omega$, we can form a CPR graph for this representation, which stands for ``C-group Permutation Representation" graph.  This concept extends naturally to any group generated by involutions, and is defined as follows.  Let $f$ be an embedding of $\G(\calP) = \langle \rho_0, \rho_1, \rho_2 \rangle$ into the symmetric group $\Sym(\Omega)$.  The {\em CPR graph} $X$ of $\calP$ given by $f$ is a 3-edge-labeled multigraph with nodes $\Omega$ such that for any $i,j \in \Omega$ with $i \neq j$, there is a single $ij$ edge of $X$ of label $k$ if and only if the image of $\rho_k$ under $f$ sends $i$ to $j$.   When $\G(\calP)$ acts faithfully on its vertex set, we can thus construct a {\em vertex CPR graph} for the regular polyhedron $\calP$, where the vertices of the graph $X$ are themselves vertices of $\calP$.  

Given a CPR graph $X$, and any subset $I$ of $\{0,1,2\}$, we can construct the subgraph $X_I$ which denotes the spanning subgraph of $X$, including all the vertices of $X$, whose edge set consists of the edges with labels $k \in I$.   Much is known about the structure of these subgraphs $X_I$.  For instance, if $I = \{ 0 ,2 \}$ then every connected component of $X_I$ is either a single vertex, a single edge, a double edge, or an alternating square.   The type of the polyhedron also can easily be determined from the CPR graph.  Let $X$ be a CPR graph for a regular polyhedron of type $\{p,q\}$.  Then, if $I = \{0,1\}$ (or dually $\{1,2\}$) then the connected components of $X_I$ are single vertices, double edges, alternating paths, or alternating cycles, and $p$ (or dually $q$) is the least common multiple of the number of vertices in each alternating path and half the number of vertices in each alternating cycle.

	
	\subsection{Vertex-faithful polyhedra}


	We say that a regular polyhedron is \emph{vertex-faithful} if its automorphism group acts faithfully
	on its vertices; in other words, if the only automorphism that fixes every vertex is
	the identity. We note that if a polyhedron is \emph{vertex-describable}, meaning that each face
	is completely determined by its vertex-set, then it must also be vertex-faithful. This is because,
	if an automorphism fixes every vertex of a vertex-describable polyhedron, it must fix every face and
	thus must be the identity (since the automorphism group of a polyhedron acts freely on the flags). 
	The converse is not true; the hemi-cube $\{4,3\}_3$ is an example of a vertex-faithful
	polyhedron that is not vertex-describable.
	
	If $\calP$ is a regular polyhedron with $\G(\calP) = \langle \rho_0, \rho_1, \rho_2 \rangle$ and with base vertex
	$u$, then the stabilizer of $u$ is $\langle \rho_1, \rho_2 \rangle$, and the stabilizer of an arbitrary
	vertex $u \varphi$ is $\varphi^{-1} \langle \rho_1, \rho_2 \rangle \varphi$, where $\varphi \in \G(\calP)$. Then the kernel of the
	action of $\G(\calP)$ on the vertices is
	\[ \bigcap_{\varphi \in \G(\calP)} \varphi^{-1} \langle \rho_1, \rho_2 \rangle \varphi, \]
	which is the normal core of $\langle \rho_1, \rho_2 \rangle$ in $\G(\calP)$ (also called simply
	the core of $\langle \rho_1, \rho_2 \rangle$ in $\G(\calP)$). Thus, $\calP$ is
	vertex-faithful if and only if $\langle \rho_1, \rho_2 \rangle$ is core-free in $\G(\calP)$.
	
	In order to analyze regular polyhedra with a fixed number of vertices, we will treat flat and non-flat polyhedra
	separately. Proposition~\ref{prop:fixes-verts} below gives one of the main reasons why.
	
	\begin{proposition}[Proposition 2.2 of \cite{tight3}]
	\label{prop:simple-quo}
	Let $\G = \langle \rho_0, \rho_1, \rho_2 \rangle$ be a string C-group. Let $N = \langle (\rho_1 \rho_2)^k \rangle$ for some $k \ge 2$. If $N$ is normal in $\G$, then $\G/N$ is a string C-group.
\end{proposition}
	
	\begin{proposition}
	\label{prop:fixes-verts}
	Suppose that $\calP$ is a non-flat regular polyhedron of type $\{p, q\}$ with $v$ vertices, where $v$ is finite. Let $N$ be the normal core of
	$\langle \rho_1, \rho_2 \rangle$ in $\G(\calP)$. Then:
	\begin{enumerate}
	\item $N = \langle (\rho_1 \rho_2)^{q'} \rangle$ for some $q' \geq 3$ where $q'$ divides $q$.
	\item $\G(\calP)/N$ is the automorphism group of a non-flat, vertex-faithful regular polyhedron of type $\{p, q'\}$
	with $v$ vertices.
	\end{enumerate}
	\end{proposition}
	
	\begin{proof}
	Let $\varphi \in N$. If $\varphi$ is a reflection (that is, the product of an odd number
	of generators), then $\varphi$ is conjugate to either $\rho_1$ or $\rho_2$. Since $N$ is normal, that implies
	that $\rho_1$ or $\rho_2$ is in $N$ and thus fixes all vertices. If $\rho_1 \in N$, then $\rho_0 \rho_1 \rho_0 \in N$,
	and by the intersection condition, $\rho_0 \rho_1 \rho_0 \in \langle \rho_1 \rangle$. This forces $(\rho_0 \rho_1)^2 = 1$,
	and so $p = 2$, which would make $\calP$ flat. So, suppose $\rho_2 \in N$. Then $\G(\calP) / N$ is a quotient
	of the dihedral group $\langle \rho_0, \rho_1 \rangle$, and so $[\G(\calP) : N] \leq 2p$. (Note that $p$ must be
	finite since $v$ is finite.) Now, since $N$ is a subgroup of a dihedral group of order 
	$2q$ and does not contain $\rho_1$, the index of $N$ in $\langle \rho_1, \rho_2 \rangle$ is at least 2.
	Furthermore, the index of $\langle \rho_1, \rho_2 \rangle$ in $\G(\calP)$ is equal to $v$ which is
	at least as large as $p$. So
	\[ 2p \leq 2v \leq [\G(\calP) : \langle \rho_1, \rho_2 \rangle] [\langle \rho_1, \rho_2 \rangle : N] = [\G(\calP) : N] \leq 2p. \]
	So in this case we get $v = p$, meaning that $\calP$ is flat. So when $\calP$ is non-flat, neither $\rho_1$
	nor $\rho_2$ can lie in $N$. Thus $N$ cannot contain any
	reflections, and so $N = \langle (\rho_1 \rho_2)^{q'} \rangle$ for some $q'$ dividing $q$. 

	Next, Proposition~\ref{prop:simple-quo} shows that $\G(\calP)/N$ is the automorphism group of a regular
	polyhedron $\calQ$, which will have type $\{p, q'\}$. 
	Clearly the image of $\langle \rho_1, \rho_2 \rangle$ will be core-free, so that $\calQ$ is
	vertex-faithful. Furthermore, since $N$ was the kernel of the action on the vertices, $\calQ$
	will have the same number of vertices as $\calP$. Since $\calP$ is not flat, that implies that $v > p$,
	and thus $\calQ$ is also not flat. Finally, this implies that $q' \geq 3$, since if $q' = 2$ then $\calQ$
	has type $\{p, 2\}$ which would imply that it is flat.
	\end{proof}
	
	\begin{corollary}
	\label{cor:no-vert-faithful}
	If there are no non-flat, vertex-faithful regular polyhedra with $v$ vertices,
	then there are no non-flat regular polyhedra whatsoever with $v$ vertices.
	\end{corollary}
	
	We will see later that there is an infinite family of flat regular polyhedra with no vertex-faithful quotients, so the assumption that $\calP$ was non-flat was essential.
	
Given a vertex-faithful regular polyhedron $\calQ$, there may be infinitely many regular polyhedra $\calP$ that cover
$\calQ$ and have the same number of vertices. However, it is sometimes possible to bound the size of $\calP$.

\begin{proposition} \label{prop:vf-quos}
Suppose that $\calQ$ is a vertex-faithful regular polyhedron of type $\{p, q'\}$ and that
$\calP$ is a regular polyhedron of type $\{p, q\}$ that properly covers $\calQ$ and has the same number of vertices.
Let $\psi: \G(\calP) \to \G(\calQ)$ be the canonical covering that sends generators to generators.
\begin{enumerate}
\item If $\calQ$ is non-orientable, then $\calP$ is non-orientable.
\item In $\G(\calP)$, $\rho_0 \s_2^{q'} \rho_0 = \s_2^{aq'}$ for some $a$ satisfying $a^2 \equiv 1$ (mod $q/q'$).
\item If $\alpha$ is a word in $\G(\calP)$ such that $\alpha \psi = 1$, then $\alpha$ commutes with $\s_2^{q'}$.
\item Suppose that $\alpha \in \G(\calP)$ such that $\alpha \psi = 1$. Let $s$ be the number of occurrences (mod $2$)
of $\rho_0$ in $\alpha$, and let $t$ be the total number of occurrences (mod $2$) of $\rho_1$ and $\rho_2$ in $\alpha$.
Then $\s_2^{q'} = \s_2^{a^s (-1)^t q'}$. In particular, if $\alpha$ has odd length and $\rho_0$ occurs an even number
of times, then $q = 2q'$, and if $p$ is odd, then $\rho_0$ inverts $\s_2^{q'}$.
\end{enumerate}
\end{proposition}

\begin{proof}
Let $N = \ker \psi = \langle \s_2^{q'} \rangle$.
If $\calQ$ is non-orientable, that means that there is a word $\alpha$ of odd length in $\G(\calP)$ such that
$\alpha \in N$. But the elements of $N$ all have even length, and so $\calP$ is non-orientable. That proves part (a).

For part (b), the first part is obvious since $N$ is normal. Then since $\rho_2$ inverts $\s_2$,
\[ \s_2^{q'} = (\rho_0 \rho_2)^2 \s_2^{q'} (\rho_2 \rho_0)^2 = \s_2^{a^2 q'}, \]
proving the second part.

Part (c) is clear since in this case $\alpha \in N$.

To prove part (d), first recall that $\rho_1$ and $\rho_2$ both invert $\s_2^{q'}$. The result then follows immediately
from parts (b) and (c) considering the equation $\alpha^{-1} \s_2^{q'} \alpha = \s_2^{q'}$ and expanding the left-hand
side.
\end{proof}

We can say more about the covering of vertex-faithful regular polyhedra depending on the lengths of their $j-${\em holes} and $j-${\em zigzags} (see Section 7B of~\cite{ARP} for a combinatorial description).  Algebraically, the length of a 1-hole, $p$, is the order of the element $\rho_0 \rho_1$, where the length of a 2-hole is the order of the element $h=\rho_0 \rho_1 \rho_2 \rho_1$.  The length of a 1-zigzag is the size of the Petrie polygon of the map, which is the order of the element $z_1=\rho_0 \rho_1 \rho_2$.  Finally, the order of the element $z_2=\rho_0 \rho_1 \rho_2 \rho_1 \rho_2$ gives the length of the 2-zigzags of $\calP$.  

\begin{corollary} \label{cor:odd-zigzags}
Suppose that $\calQ$ is a vertex-faithful regular polyhedron of type $\{p, q'\}$ and that
$\calP$ is a regular polyhedron of type $\{p, q\}$ that properly covers $\calQ$ and has the same number of vertices.
Let $h$, $z_1$, and $z_2$ be the elements of $\G(\calQ)$ as described above. 
If either $p$ or $|h|$ is odd and either $|z_1|$ or $|z_2|$ is odd, then $q = 2q'$.
In particular, $|\G(\calP)| = 2|\G(\calQ)|$.
\end{corollary}

\begin{proof}
In this case, part (d) says that $\rho_0$ commutes with $\s_2^{q'}$ and inverts $\s_2^{q'}$.
It follows that $\s_2^{-q'} = \s_2^{q'}$, and so $q = 2q'$. Then $|\G(\calP)| = 2qv = 4q'v = 2|\G(\calQ)|$.
\end{proof}

	The following proposition is an adaptation of \cite[Cor. 13]{maps-few-faces}.

	\begin{proposition} \label{prop:steves-thm}
	Let $\calP$ be a regular polyhedron of type $\{p, q\}$. For every $q' < q/2$ such that $q'$ divides $q$, the
	number of vertices fixed by $(\rho_1 \rho_2)^{q'}$ is a divisor of the total number of vertices.
	\end{proposition}
	
	\begin{proof}
	Let $u$ be the base vertex, with stabilizer $\langle \rho_1, \rho_2 \rangle$,
	and fix a $q' < q/2$ such that $q'$ divides $q$. If $v$ is an arbitrary vertex
	of $\calP$, then we may write $v = u \alpha$ for some $\alpha \in \G(\calP)$.
	Then $(\rho_1 \rho_2)^{q'}$ fixes $v$ if and only if $\sigma := \alpha (\rho_1 \rho_2)^{q'} \alpha^{-1}$ fixes
	$u$, which is true if and only if $\sigma \in \langle \rho_1, \rho_2 \rangle$. Now, since $q' < q/2$, it follows
	that $(\rho_1 \rho_2)^{q'}$ has order 3 or more, and thus the same is true of $\sigma$. Then since
	$\sigma \in \langle \rho_1, \rho_2 \rangle$, which is dihedral, we find that 
	$\sigma \in \langle (\rho_1 \rho_2)^{q'} \rangle$. Therefore, $(\rho_1 \rho_2)^{q'}$ fixes $u \alpha$
	if and only if $\alpha$ normalizes $\langle (\rho_1 \rho_2)^{q'} \rangle$. 
	
	Now, let $N$ be the normalizer of $\langle (\rho_1 \rho_2)^{q'} \rangle$ in $\G(\calP)$, and let
	$H = \langle \rho_1, \rho_2 \rangle$. The stabilizer of $u$ is $H$, and so $(\rho_1 \rho_2)^{q'}$ fixes $u \alpha$
	if and only if it fixes $u \beta \alpha$ for every $\beta \in H$. By the previous paragraph,
	this means that $\alpha \in N$ if and only if $\beta \alpha \in N$ for every $\beta \in H$. It follows that the
	number of vertices fixed by $(\rho_1 \rho_2)^{q'}$ is the index of $H$ in $N$. Since the total number of vertices
	is the index of $H$ in $\G(\calP)$, the result follows.
	\end{proof}
	
All the calculations in this paper were verified in GAP~\cite{GAP} and \textsc{Magma}~\cite{Magma}.  When these programs were used for a calculation, we will say that they were done on a Computer Algebra System (CAS).


\section{Polyhedra with few vertices}
\label{sec:few}

In this section we consider vertex-faithful regular polyhedra with few vertices.  If a regular polyhedron has vertex figures of size $q$, this means that the base vertex is incident to $q$ edges.  We will refer to the other vertices incident to these $q$ edges as the {\em vertices on the base vertex figure}.
 For a fixed size of vertex figure $q$, we will show that the number of vertices $v$ is bounded below by $q$, and we will consider what happens when $v$ obtains this minimum.   We note that the same bound does not apply for arbitrary regular polyhedra; for example the regular map $\{3,6\}_{(2,0)}$ is a polyhedron with only four vertices.  

In this section we also classify the vertex-faithful regular polyhedra with fewer than sixteen vertices.  This classification will be leveraged in the later sections.

\begin{proposition} \label{prop:qorientable}
Suppose $\calP$ is a vertex-faithful regular polyhedron of type $\{p, q\}$. If $\calP$ is orientable, then $q < v$; if $\calP$ is non-orientable, then $q \leq v$.
\end{proposition}
		
\begin{proof}
Let $\G(\calP) = \langle \rho_0, \rho_1, \rho_2 \rangle$, 
and let $H = \langle \rho_1, \rho_2 \rangle$.  The vertices of $\calP$ correspond to cosets $H \varphi$, with the base vertex corresponding to the coset $H = H1$ (where $1$ is the identity of $\G(\calP)$).   We may assume that $v \geq 4$.  Otherwise, if $\calP$ is vertex-faithful, then $\G(\calP)$ embeds into the symmetric group acting on three elements, and thus has size less than or equal to six.  It can be checked (for instance in~\cite{atlas}) that there is no regular polyhedron with an automorphism group this small.

Let $\s_2 = \rho_1 \rho_2$, and consider the action of $\s_2$ on the vertex $H \rho_0$, 
which is the other vertex incident to the base edge. 

Assume that $q \geq v$.  There are at most $v-1$ distinct vertices on the base vertex figure,
and thus the stabilizer of 		
$H \rho_0$ in $\langle \s_2 \rangle$ is non-trivial and generated by $\s_2^a$ for some smallest positive integer $a$.  Now, to say that $\s_2^a$ fixes $H \rho_0$ is to say that $\rho_0 \s_2^a \rho_0$ lies in $H = \langle \rho_1, \rho_2 \rangle$.
If $\calP$ is orientable, or if $\s_2^a$ has order 3 or more, then $\rho_0 \s_2^a \rho_0$ must be
a power of $\s_2^a$. In this case, $\langle \s_2^a \rangle$ is normal. 
Then for any vertex $H \varphi$, we have that $H \varphi \s_2^a = H \s_2^{ak} \varphi = H \varphi$,
and so $\s_2^a$ fixes every vertex of $\calP$, and thus $\calP$ is not vertex-faithful. Thus in particular, if $\calP$
is orientable and vertex-faithful, then $q < v$.

The remaining case to consider is when $\calP$ is non-orientable, $a = q/2$, and $\rho_0 \s_2^{q/2} \rho_0$ is not in 
$\langle \s_2^{q/2} \rangle$. Then $\rho_0 \s_2^{q/2} \rho_0 = \alpha \in \langle \rho_1, \rho_2 \rangle$, where $\alpha$ has odd length.
We note that $\rho_2$ commutes with the left side, so we have $(\rho_2 \alpha)^2 = 1$. 
Then since $\rho_2 \alpha$ has even length and order 2, 
it follows that $\rho_2 \alpha = \s_2^{q/2}$, and so $\alpha = \rho_2 \s_2^{q/2}$. 
In other words, the following relation holds in $\G(\calP)$:

\begin{equation} \label{rel1}
\rho_0 \s_2^{q/2} \rho_0 = \rho_2 \s_2^{q/2} 
\end{equation}

It follows that $\s_2^{q/2} \rho_0 \s_2^{q/2} = \rho_0 \rho_2$. Furthermore,
$\s_2^{q/2} \rho_0 \rho_1 \s_2^{q/2} = \rho_0 \rho_2 \rho_1$. This means that $\calP$
is \emph{internally self-Petrie} (see~\cite[Sec. 6]{internal-duality}). This also implies that $(\calP^{\delta})^{\pi}$ is internally self-dual.

Consider what Relation~\ref{rel1} means for the $\frac{q}{2}$ vertices on the base vertex figure, where $\s_2^\frac{q}{2}$ acts trivially.  It implies that if any neighbor of the base vertex is also fixed by $\rho_0$, then it is fixed by $\rho_2$.  
Thus $\rho_0$ has at most one fixed point on the base vertex figure (with zero fixed points when $\frac{q}{2}$ is odd). See Figure~\ref{fig:cpr}.
\begin{center}
\begin{figure}[htbp]
$$\includegraphics[scale=.6]{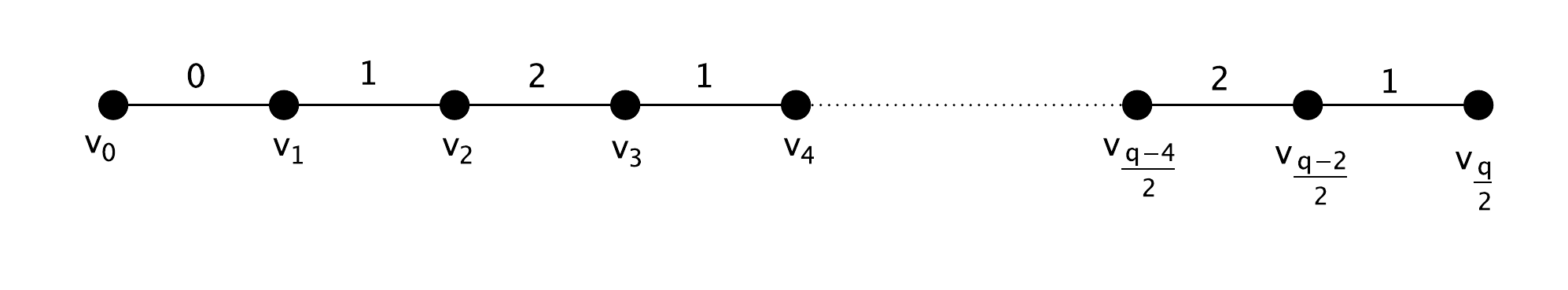}$$ \\
$$\includegraphics[scale=.6]{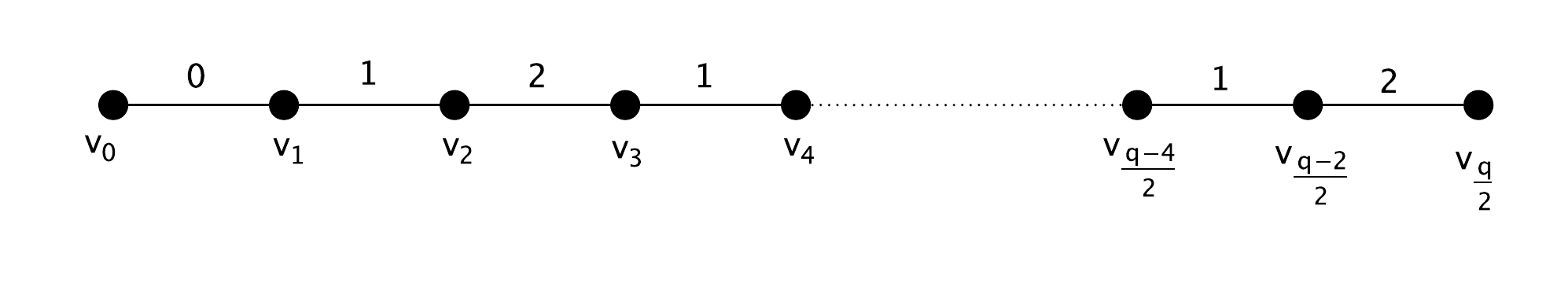}$$
\caption{Subgraphs of vertex CPR graphs induced by the base vertex $v_0$ and the $\frac{q}{2}$ vertices on the base vertex figure, for $\frac{q}{2}$ even and odd respectively.}
\label{fig:cpr}
\end{figure}
\end{center}
Let $X$ be the vertex CPR graph for $\calP$, and let $w$ be any vertex on the base vertex figure not fixed by $\rho_2$, and thus not fixed by $\rho_0$.  
Since $\rho_0$ and $\rho_2$ commute, in $X$ the vertex $w$ must be part of either an alternating square with edges of labels 0 and 2, or an end point of a double edge of labels 0 and 2.  
Since $\s_2^{q/2}$ fixes $w$, Relation~\ref{rel1} shows that the double edge case is not possible, and thus $w$ and $(w) \rho_2$ are part of an alternating square with two other vertices $x$ and $(x) \rho_2$.  
Similarly, Relation~\ref{rel1} implies that $x$ and $(x) \rho_2$ are interchanged by $\s_2^{q/2}$, and so they are not on the base vertex figure.  
Therefore for every pair of vertices $w$ and $(w) \rho_2$ on the base vertex figure, 
there exists a pair of vertices $x$ and $(x) \rho_2$ not on the base vertex figure. 
Counting up the vertices, we have the base vertex, the $\frac{q}{2}$ vertices on the base vertex figure, 
and the at least $\frac{q}{2}-2$ (or $\frac{q}{2} -1$ if $\frac{q}{2}$ is odd) vertices that are sent to the base vertex figure by $\rho_0$.  
Thus $v \geq 1+\frac{q}{2}+\frac{q}{2}-2 = q-1$ (or $v\geq q$ when $\frac{q}{2}$ is odd).

If $q = v+1$, then $\frac{q}{2}$ is even. Furthermore, if $\frac{q}{2}$ is even,
then exactly $\frac{q}{2}-2$ vertices are moved by $\s_2^{q/2}$ and sent to the base vertex figure by $\rho_0$.  We will now rule out the possibility that $q \geq v$ when $\frac{q}{2}$ is even, which will then prove that $q \leq v$.

Consider the action of $\s_2^{q/4}$ on the $\frac{q}{2}-2$ vertices that are moved by $\s_2^{q/2}$. Since $\s_2^{q/2}$
interchanges these vertices in pairs, $\s_2^{q/4}$ must act as $4$-cycles on all of these vertices.
Thus $\frac{q}{2}-2$ is divisible by $4$, which implies that $q \equiv 4$ (mod $8$).
So we may write $q = 4q'$ where $q'$ is odd. In order for $\s_2$ to have
order $q$, there must be some orbit of size $4k$. Furthermore, since the $q/2$ vertices on the base vertex figure
form a single orbit of size $2q'$, the orbit of size $4k$ must be among vertices that are not
fixed by $\s_2^{q/2}$, all of which are sent to the base vertex figure by $\rho_0$. Call this orbit $T$.
In the subgraph $X_{\{1,2\}}$ the vertices of $T$ form either an alternating path or cycle.  
In either case you can label the vertices of $T$ as $(0,1,2,3,\ldots,4k-1)$ based on the order in $X_{\{1,2\}}$, 
where vertex $i$ is adjacent to $i+1$ for $0 \leq i \leq 4k-2$.

\begin{figure}
$$\includegraphics[scale=.4]{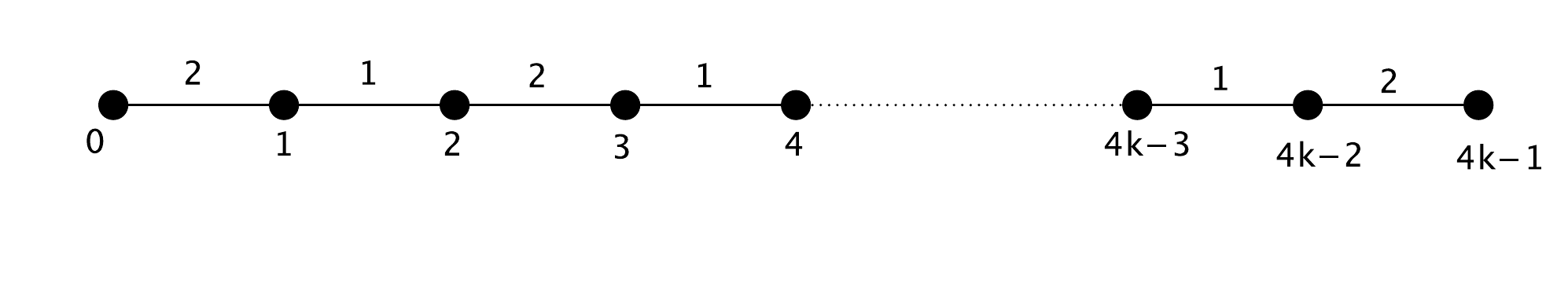}$$
$$\includegraphics[scale=.4]{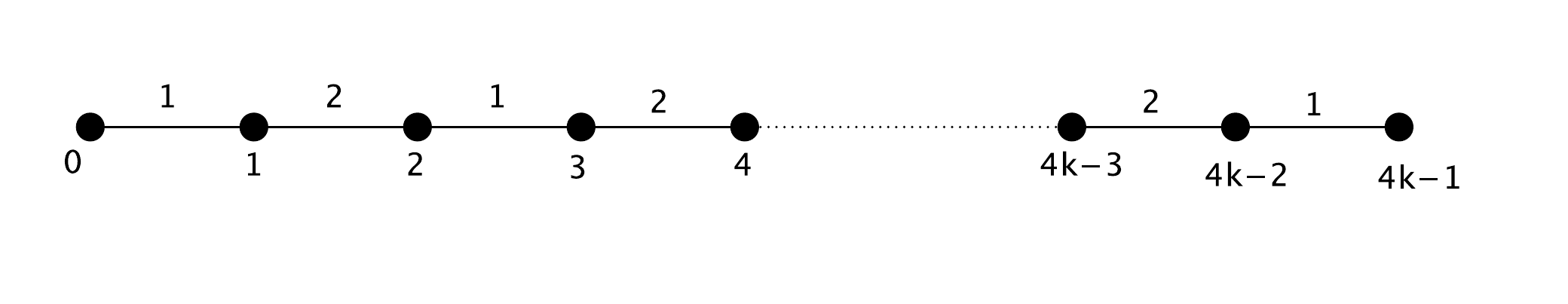}$$
$$\includegraphics[scale=.4]{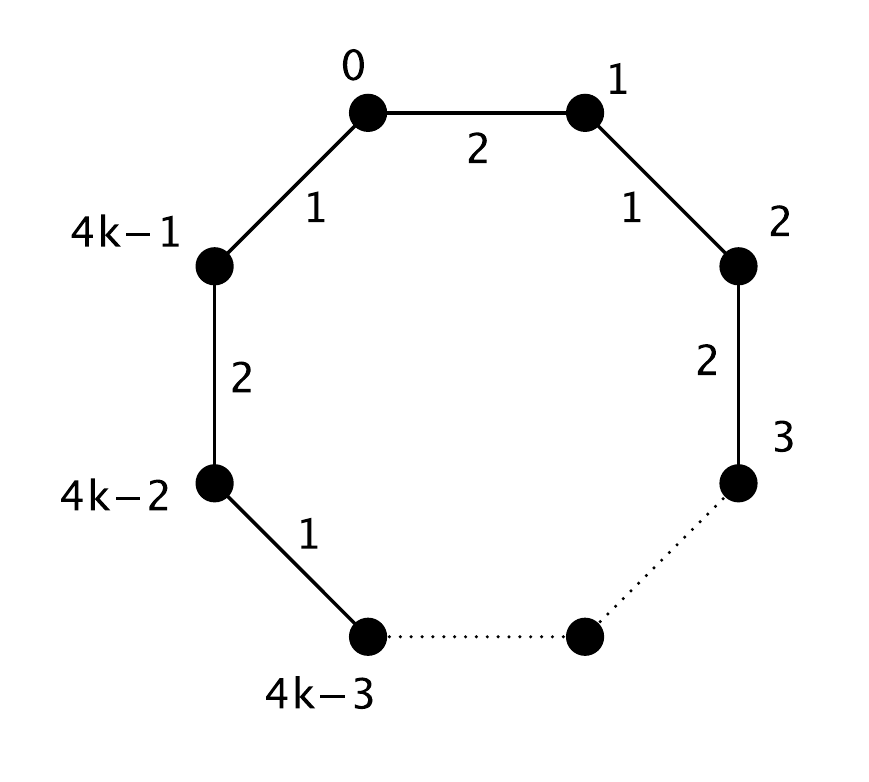}$$

\caption{Possible subgraphs of the vertex CPR graph for vertices in $T$.}
\end{figure}

Using the left hand side of Relation~\ref{rel1}, 
each vertex in $T$ is fixed by $\rho_0 \s_2^{\frac{q}{2}} \rho_0$.   On the other hand, for all vertices $i$,  $(i) \s_2^{\frac{q}{2}} = 4k-1-i$ when $T$ is an alternating path, and $(i) \s_2^{\frac{q}{2}} = i \pm 2k$ when $T$ is an alternating cycle.  Furthermore, in both cases $(i) \rho_2 \equiv i \pm 1 \pmod {4k}$.  In either case, $\rho_2 \s_2^{\frac{q}{2}}$
cannot fix all vertices of $T$, and so Relation~\ref{rel1} would not hold.   Therefore, $q \geq v$ and $\frac{q}{2}$ even is not possible for any vertex-faithful regular polyhedron.  In particular, for any vertex-faithful regular polyhedron $q \leq v$ and $q = v$ only if $\frac{q}{2}$ is odd and $\calP$ is internally self-Petrie and non-orientable.
\end{proof}

We saw that if a vertex-faithful regular polyhedron has the property that $q=v$,  then $v$ is  even, $\frac{v}{2}$ is odd, $\calP$ is non-orientable, self-Petrie, and satisfies Relation~\ref{rel1}.  This is extremely restrictive, and will allow us to fully understand the structure of $\calP$ up to isomorphism.    In this case, the orbits of the vertices under the action of $\langle \rho_1 \rho_2 \rangle$ are of size 1  in the case of the base vertex, size $\frac{v}{2}$ for the vertices on the base vertex figure.  
Additionally, every vertex not on the base vertex figure, is sent to the base vertex figure by $\rho_0$.

Let $T$ be any $\langle \rho_1, \rho_2 \rangle$-orbit of a vertex not on the base vertex figure (other than the base vertex).  We will show that $T$ always consists of two vertices connected by a single edge labeled $2$. Assume to the contrary that $T$ has at least 3 vertices.  Since $\rho_2$ does not fix any of these vertices, there must be an even number $2j$ of vertices in $T$, with $j > 1$. 
Similar to in the previous proof,  
each vertex of $T$ is fixed by $\rho_0 \s_2^{\frac{q}{2}} \rho_0$.   On the other hand, for all vertices $i$,  $(i) \s_2^{\frac{q}{2}} = 2j-1-i$ when $T$ is an alternating path, and $(i) \s_2^{\frac{q}{2}} = i \pm j$ when $T$ is an alternating cycle.  Furthermore, in both cases $(i) \rho_2 \equiv i \pm 1 \pmod {2j}$.  In either case, $\rho_2 \s_2^{\frac{q}{2}}$
cannot fix all vertices of $T$, and so Relation~\ref{rel1} would not hold if $T$ had at least 3 vertices. So $T$ has two vertices
that are connected by an edge labeled $2$. If they were also connected by an edge labeled $1$, then $\rho_2 \s_2^{\frac{q}{2}}$ would
move both vertices whereas $\rho_0 \s_2^{\frac{q}{2}} \rho_0$ would fix both vertices, violating Relation~\ref{rel1}. Thus $T$ consists of a single edge labeled $2$.

The vertex-set of the base facet consists of the $\langle \rho_0, \rho_1 \rangle$-orbit of the base vertex. We have that $\rho_0$
sends the base vertex to one of its neighbors; then $\rho_1$ sends that vertex to another neighbor of the base vertex,
and then $\rho_0$ sends that vertex to a vertex not on the base vertex figure. The above argument demonstrates that
$\rho_1$ fixes this vertex, and so there are exactly 4 vertices incident to the base facet, and thus $\calP$ is of type $\{4,v\}$. 
This gives a unique possible regular polyhedron for each $v \equiv 2$ (mod $4$).
  The vertex CPR graphs for $v=6$ and $v=10$ are shown below.   Using Theorem 4.4 of~\cite{CPR} we see that these are in fact polyhedra, but it remains to show that this permutation representation is the action on the vertices.

$$\includegraphics[scale=.5]{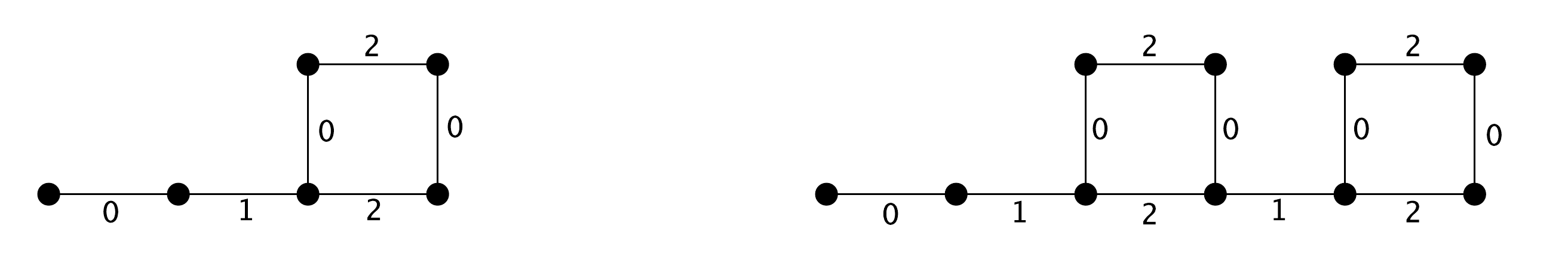}$$
	
\begin{corollary} \label{cor:vq}
If $\calP$ is a vertex-faithful regular polyhedron of type $\{p, q\}$ with $q = v$, then $\calP$ is $(\{4,4\}_{(q/2,0)})^{\pi \delta}$, with $q \geq 6$ and $q/2$ odd.
\end{corollary}

\begin{proof}

We have shown above that there is only one possibility for such a vertex-faithful regular polyhedron of type $\{p,q\}$ with $q=v$ to exist.  Furthermore, we have seen that if it does exist, then $p=4$ and $v \equiv 2$ (mod $4$).  It remains to show that such a polyhedron does exist for these cases. The result can be checked for $v \leq 5$ using a CAS.

Fix $q$ such that $q \geq 6$ and $q \equiv 2$ (mod $4$), and let $\calP$ be the regular polyhedron with the 
vertex CPR graph above. 
We have established that $\calP$ is a polyhedron of type $\{4, q\}$ that satisfies Relation~\ref{rel1}.
In particular, this means that $\calP$ is self-Petrie and thus a quotient of the universal polyhedron $\calQ = \{4, q\}_4$.
Now, $\calQ$ is the dual of the Petrial of $\{4, 4\}_q$, the latter of which is isomorphic to the toroidal map $\{4, 4\}_{(q/2, q/2)}$
and which has an automorphism group of order $4q^2$. Since $\calQ$ is orientable whereas $\calP$ is non-orientable,
it follows that $\calP$ is a proper quotient of $\calQ$. Then the Petrial of the dual of $\calP$ is a proper smooth quotient
of $\{4, 4\}_q$ that still has petrie polygons of length $q$. The only possibility is $\mathcal T = \{4, 4\}_{(q/2, 0)}$,
and so $\calP$ is the dual of the Petrial of $\mathcal T$.
\end{proof}

Now let us consider non-vertex-faithful regular polyhedra whose vertex-faithful quotient is $(\{4,4\}_{(q/2,0)})^{\pi \delta}$.

\begin{lemma}\label{lem:nocover}
Suppose $\calP$ is a finite non-flat vertex-faithful regular polyhedron of type $\{p,q\}$ which satisfies Relation~\ref{rel1}, with $q$ even. 
If $\calQ$ is a regular polyhedron that covers $\calP$ and $\calQ$ has the same number of vertices as $\calP$, then $\calQ = \calP$.
In particular, there are no non-vertex-faithful regular polyhedra whose vertex-faithful quotient is $(\{4,4\}_{(q/2,0)})^{\pi \delta}$ with $q \geq 6$ and $q/2$ odd.
\end{lemma}

\begin{proof}
Let $\calP$ and $\calQ$ satisfy the conditions of the lemma. Let $\langle \rho_0, \rho_1, \rho_2 \rangle = \Gamma(\calQ)$, and let $\s_2 = \rho_1 \rho_2$.
By Proposition~\ref{prop:fixes-verts}, $N = \langle \s_2^q \rangle$, and thus lifting Relation~\ref{rel1} from $\Gamma(\calP)$ to $\Gamma(\calQ)$ we get $\rho_0 \s_2^{\frac{q}{2}} \rho_0 \s_2^{-\frac{q}{2}} \rho_2 = \s_2^{kq}$ for some integer $k$.

Therefore, in $\Gamma(\calQ)$
$$ \rho_0 \s_2^{\frac{q}{2}} \rho_0  = \s_2^{kq} \rho_2 \s_2^{\frac{q}{2}}.$$

The right hand side of this equation has order 2 (being an element of odd length in a dihedral group), and thus $\s_2^{\frac{q}{2}}$ also has order 2.  Therefore $\calQ$ is of type $\{p,q\}$, and since $\calP$ was also of this type we know that $N$ was trivial.  Thus $\calP = \calQ$. The rest follows immediately.
\end{proof}

Motivated by the preceding results, we briefly consider vertex-faithful regular polyhedra with $q=v-1$.  We leave the full details of this case as an open question. Up to 2000 flags, there are only 7 regular polyhedra with $q = v-1$, of types $\{3, 3\}$, $\{4, 3\}$, $\{3, 5\}$, $\{5, 5\}$,
$\{4, 15\}$, $\{8, 15\}$, and $\{12, 15\}$. Only the first 4 are vertex-faithful.
Suppose $\calP$ is a vertex-faithful regular polyhedron with $q = v-1$.  If $\calP$ is orientably regular, then the vertices on the base vertex figure are all distinct.   This implies that the edge graph of $\calP$ is the complete graph on $v$ vertices, (which is to say, $\calP$ is a \emph{neighborly} polyhedron), as the base vertex is adjacent to every other vertex.  Then, one can use the fact that the regular imbeddings of the complete graphs in orientable surfaces have been classified in~\cite{complete-imbeddings}.

Finally, we mention one further corollary of Proposition~\ref{prop:qorientable} that helps with searching for vertex-faithful regular polyhedra.

\begin{corollary}
A vertex-faithful polyhedron with $v$ vertices has at most $2v^2$ flags.
\end{corollary}

\subsection{Vertex-faithful polyhedra with few vertices }
\label{sec:smallV}
 
Here we will give details of all the vertex-faithful polyhedra with fewer than sixteen vertices. 
The results can be achieved by analyzing the classification of all regular polytopes with up to 4000 flags~\cite{conder-atlas}, or by classifying regular polytopes for transitive groups up to degree 15.   The details of our classification are found in Table~\ref{tab:small1} and Table~\ref{tab:small2}.

 There are no vertex-faithful regular polyhedra where the number of vertices is a prime number less than 16.  
  We will see in Section~\ref{sec:prime} that this continues to hold for all primes.
  Additionally, we point out that there are exactly two vertex-faithful regular polyhedra with $v=14$ vertices; one of them is flat and will be described in the second row of Table~\ref{tab:flat-2b}, the other is
   $(\{4,4\}_{(v,0)})^{\pi \delta}$, the dual of the Petrial of $\{4,4\}_{(v,0)}$.  We will see in Section~\ref{sec:twiceprime} that this classification remains whenever $v$ is twice any prime at least 7.

     In our tables, we provide the ``Atlas Canonical Name'' of the polyhedron from Hartley's Atlas~\cite{atlas}, as well as a description of the automorphism group of the polyhedron, where each one is a quotient of the string Coxeter group of type $\{p,q\}$, with $p$ and $q$ found in the Atlas Canonical Name.  For each group, we give the order of the Petrie element $z_1=\rho_0 \rho_1 \rho_2$, which gives the length of a 1-zigzag of $\calP$.  We also give the orders of the elements $h=\rho_0 \rho_1 \rho_2 \rho_1$ which describes the sizes of the 2-holes in the map, and $z_2=\rho_0 \rho_1 \rho_2 \rho_1 \rho_2$ which gives the length of the 2-zigzag of $\calP$.  

In Problem 7 of~\cite{polytope-problems}, it is asked to what extent a regular polyhedron of fixed type is determined by the lengths of its $j-$holes and the lengths of its $j-$zigzags.   Considering this question, we also note whether the polyhedron is ``universal'' with respect to the length of its $1$-zigzags, $2$-zigzags, and $2$-holes; otherwise further relators are needed to form a presentation for the automorphism group.

\begin{table}[h]
\begin{center} \begin{tabular}{|l|l|l|l|l|l|l|} \hline
$v$ & Atlas Canonical Name & Also Known As &   $|z_1|$ & $|h|$ & $ | z_2 |$ & Universal\\ \hline
4 & $\{3,3\}*24$ & tetrahedron, 3-simplex, $\{3,3\}$ & 4 & 3 & 4  &Y\\ 
4 &$\{4,3\}*24$ &  hemi-cube, $\{4,3\}_3$ & 3 & 4 & 3 & Y \\  \hline
 6& $\{3,4\}*48$ & octahedron, $\{3,4\}$ &6&4 & 4& Y\\ 
6& $\{3,5\}*60$ &  hemi-icosahedron, $\{3,5\}_5$ &5&5 & 3&Y \\
6& $\{4,6\}*72$ & $(\{4,4\}_{(3,0)})^{\pi \delta}$  &4&6 & 6  &N \\
6& $\{5,5\}*60$ &   $\{5,5\}_3$ & 3 &3 &5&Y\\
6& $\{6,3\}*36$ &   $\{6,3\}_{(1,1)}$ & 6 &6 & 6 &N \\
6& $\{6,4\}*48b$ &   $\{6,4\}_3$ &  3& 4 & 4  &Y\\ \hline
8&  $\{4,3\}*48$ & cube, $\{4,3\}$ & 6 &4& 6  &Y\\ 
8&  $\{4,4\}*64$ &  $\{4,4\}_{(2,2)}$, $\{4,4\}_4$,  &4&4& 4& Y\\
8&  $\{6,3\}*48$ & $\{6,3\}_{(2,0)}$, $\{6,3\}_4$ &4&6& 4   &Y\\
8&  $\{8,4\}*64b$ &  &8&4& 4 &N\\
\hline
9&   $\{3,6\}*108$ &$\{3,6\}_{(3,0)}$, $\{3,6\}_6$,  &6&6& 6  &Y\\
9&  $\{4,4\}*72$ &  $\{4,4\}_{(3,0)}$, $\{4,4 | 3\}$,  &6&3& 6 &Y\\
9&  $\{6,4\}*72$ &  $(\{4,4\}_{(3,0)})^\pi$ &4&6&3     &Y\\
9&  $\{6,6\}*108$ &$\{6,6\}_3$   &3&6& 6&Y\\
\hline
10&   $\{4,6\}*120$ & &5&3 & 4 &Y\\ 
10&   $\{4,10\}*200$ & $(\{4,4\}_{(5,0)})^{\pi \delta}$ & 4& 10 &   10 &N\\
10&   $\{5,3\}*60$ & hemi-dodecahedron, $\{5,3\}_5$  & 5& 5& 5  &Y\\
10&   $\{5,6\}*120a$ &   &4  &4&  3 &Y\\   
10&   $\{6,6\}*120$ &  & 6& 5&  5 &N \\
10&   $\{10,5\}*100$ &   & 10& 10& 10 &N\\
\hline
\end{tabular}
\caption{The vertex-faithful regular polyhedra with 10 or fewer vertices.
\label{tab:small1}
}
\end{center}
\end{table}


\begin{table}
\begin{center} \begin{tabular}{|l|l|l|l|l|l|l|} \hline
$v$ & Atlas Canonical Name & Also Known As &   $|z_1|$ & $|h|$ & $ | z_2 |$ & Universal\\ \hline
12&   $\{3,5\}*120$ & icosahedron, $\{3,5\}$ &10& 5& 6 & Y  \\ 
12&   $\{3,6\}*144$ &  $\{3,6\}_{(2,2)}$ & 12&6&4 & Y  \\
12&   $\{3,8\}*192$ & $\{3,6\}_6$  &6&8& 8 & Y \\
12&   $\{4,5\}*120$ &   &6&6 & 4 & N \\
12&   $\{4,6\}*144$ &   $\{4,6\}_4$ &4&6& 6 & Y   \\
12&   $\{5,5\}*120$ &   $\{5,5 | 3\}$ &6&3 & 10 & Y \\
12&   $\{6,4\}*96$ &  & 6& 4&4 & N \\
12&   $\{6,5\}*120a$ &    &4&4& 6 & N  \\
12&   $\{6,5\}*120b$ &    &10& 10& 6 & N \\
12&   $\{6,5\}*120c$ &    &5& 10& 3 & Y\\
12&    $\{6,8\}*192a$ & $\{6,8\}_3$ &3&8& 8 & Y  \\ 
12&   $\{10,5\}*120a$ &  & 6&6 & 10 & N \\
12&   $\{10,5\}*120b$ & $\{10,5\}_3$ & 3& 6 & 5 & Y \\
12&   $\{12,4\}*96c$ &   &12&4&    4 & N \\
12&   $\{12,6\}*144c$ &  & 12& $ 6 $ & 6 & N  \\
12&   $\{12,6\}*144d$ & $ (\{3,6\}_{(2,2)} )^\pi$ &3&  4  & 6 & Y \\
\hline
14&    $\{4,14\}*392$ & $(\{4,4\}_{(7,0)})^{\pi \delta}$& 4& 14&   14 & N\\ 
14&   $\{14,7\}*196$ & &14& 14  & 14 & N \\
\hline
15&   $\{3,10\}*300$ & $\{3,10\}_6$& 6& 10&10 & Y   \\
15& $\{5,4\}*120$ &  &6& 6 & 3 & Y  \\
15& $\{6,4\}*120$ & &5&3& 6 & Y   \\
15& $\{6,10\}*300$ &  $\{6,10\}_3$  & 3&10 & 10 & Y     \\
\hline

\end{tabular}
\caption{The vertex-faithful regular polyhedra with 12, 14, or 15 vertices.
\label{tab:small2}}
\end{center}
\end{table}


\section{Polyhedra with a prime number of vertices}	
\label{sec:prime}

	In this section, we will fully classify the regular polyhedra with a prime number of vertices.

\subsection{Flat polyhedra}	
	
	If $\calP$ is a flat regular polyhedron with $b$ vertices, then it has type
	$\{b, q\}$ for some $q$. When $b = 2$, the polyhedron $\{2, q\}$ is a flat polyhedron with $2$
	vertices for every $q \geq 2$ (including $q = \infty$). If $b$ is an odd prime,
	then there are exactly two flat orientably regular polyhedra
	with $b$ vertices  \cite[Thms. 3.3 and 3.4]{tight2}. 
	The first is the universal polyhedron of type $\{b, 2\}$. The second
	has type $\{b, 2b\}$, and the automorphism group is the quotient of $[b, 2b]$ by the extra relation
	$(\rho_0 \rho_1 \rho_2 \rho_1 \rho_2)^2 = 1$. This is equivalent to the relation $\s_1 \s_2^{-1} \s_1 \s_2^3 = 1$,
	which yields $\s_2^{-1} \s_1 = \s_1^{-1} \s_2^{-3}$. Thus the flat regular polyhedron of type
	$\{b, 2b\}$ has automorphism group $\Lambda(b,2b)_{-1,-3}$ (see Section~\ref{sec:flat}).

	Additionally, the only flat non-orientably regular polyhedron with a prime number of vertices is the hemi-octahedron $\{3,4\}_3$~\cite[Thm. 5.10]{tight3}.
	
	All of the flat regular polyhedra with a prime number of vertices are summarized in Table~\ref{tab:flat-b}.
	We list the relations needed to define the automorphism group as a quotient of the string Coxeter group
	$[p, q]$ that is indicated by the type.
	Note that none of the flat regular polyhedra with a prime number of vertices is vertex-faithful.
	Indeed, in every case, $\rho_2$ fixes every vertex.
	
	\begin{table}
	\begin{center} \begin{tabular}{|l|l|l|l|l|} \hline
	Type & Relations & Orientable & Vertex-faithful & Notes \\ \hline
	$\{2, q\}$ & & Y & N & $2 \leq q \leq \infty$ \\ \hline
	$\{3, 4\}$ & $(\rho_0 \rho_1 \rho_2)^3 = 1$ & N & N & \\ \hline
	$\{b, 2\}$ & & Y & N & \\ \hline
	$\{b, 2b\}$ & $\s_2^{-1} \s_1 = \s_1^{-1} \s_2^{-3}$ & Y & N & \\ \hline
	\end{tabular}
	\caption{Flat regular polyhedra with a prime number of vertices $b$}
	\label{tab:flat-b}
	\end{center}
	\end{table}

\subsection{Non-flat polyhedra}	
	
	To classify non-flat regular polyhedra with $b$ vertices, we start by looking for vertex-faithful
	polyhedra (bearing in mind Corollary \ref{cor:no-vert-faithful}).
	
	\begin{proposition} \label{prop:no-cyclic}
	Suppose that $\calP$ is a non-flat, vertex-faithful regular polyhedron with a prime number of vertices $b$. Then:
	\begin{enumerate}
	\item $q \leq b-1$.
	\item $\G(\calP)$ does not have a normal Sylow $b$-subgroup.
	\item $\G(\calP)$ acts doubly transitively on the vertices.
	\end{enumerate}
	\end{proposition}
	
	\begin{proof}
	By Proposition~\ref{prop:qorientable}, $q \leq b$. Since $|\G(\calP)| = 2qb = 4e$, it follows that $q$ is even and thus $q \leq b-1$.
	This proves part (a).
	
	Since $b$ does not divide $q$, any Sylow $b$-subgroup of $\G(\calP)$ is cyclic and of order $b$. 
	Let $\pi$ be a $b$-cycle, and suppose that $\langle \pi \rangle$ is normal. 
	Then $\rho_1 \pi \rho_1= \pi^j$ and $\rho_2 \pi \rho_2 =\pi^k$ where $k^2=j^2 \equiv 1$ (mod $b)$. Since $b$ is prime, $j$ and $k$ are each $\pm 1$, which is to say that conjugating $\pi$ by $\rho_1$ or $\rho_2$ either fixes $\pi$ or inverts $\pi$. Label the vertices so that $\pi$ acts on them as a cycle $(0,1,2,3, \ldots, b-1)$ where $0$ is the base vertex.  Suppose that for $i=1$ or $i=2$ that $\rho_i \pi \rho_i = \pi$. Then since $\rho_i$ fixes $0$, it follows that for all $j \in \{0,\ldots, b-1\}$, 
	\[ j \rho_i = 0 \pi^j \rho_i = 0 \rho_i \pi^j = 0 \pi^j = j.\]  
	Thus $\rho_i$ would fix all vertices, which would imply that $\calP$ is not vertex-faithful. So $\rho_1 \pi \rho_1 = \rho_2 \pi \rho_2 = \pi^{-1}$. But then $(\rho_1 \rho_2) \pi (\rho_2 \rho_1) = \pi$, and by the same argument as for $\rho_1$ and $\rho_2$, this implies that $\rho_1 \rho_2$ fixes every vertex and that $\calP$ is not vertex-faithful. So $\langle \pi \rangle$ must not be normal.
	
	Part (c) follows from part (b) and Theorem~\ref{prop:group2}.
	\end{proof}
	
	\begin{lemma} \label{lemma:nonflatprime}
	There are no non-flat, vertex-faithful regular polyhedra with a prime number of vertices.
	\end{lemma}
	
	\begin{proof}
	Suppose $\calP$ is a non-flat, vertex-faithful regular polyhedron with a prime number of vertices $b$.  Let $\calP$ be of type $\{ p, q\}$ and thus $| \G(\calP) | = 2bq$.
	Since $\calP$ is not flat, $b \geq 3$.
	By Proposition~\ref{prop:no-cyclic}(c), $\G(\calP)$ acts doubly transitively on the vertices, and so
	$\langle \rho_1, \rho_2 \rangle$ (which is the stabilizer of the base vertex) acts transitively on
	the remaining $b-1$ vertices; in particular, $b-1$ divides $2q$.
	
	Let $n$ be the number of Sylow $b$-subgroups. By Proposition~\ref{prop:no-cyclic}(b), $n \neq 1$, and by the Sylow Theorems, $n \equiv 1$ (mod $b$) and $n$ divides $2q$. Since $q \leq b-1$ by
	Proposition~\ref{prop:no-cyclic}(a), this implies that $n = b+1 = 2q$. But then $b-1$ divides
	$b+1$, which is only true if $b = 3$. In that case we get $q = 2$, which yields a flat polyhedron, and indeed such a polyhedron is not vertex-faithful after all.
	\end{proof}

	\begin{theorem} \label{thm:prime-verts}
	Every regular polyhedron with a prime number of vertices $b$ is one of the flat polyhedra of type $\{b, q\}$
	included in Table~\ref{tab:flat-b}.  In particular, for each prime $b \geq 5$,  up to isomorphism there are exactly two regular polyhedra with $b$ vertices. 
	\end{theorem}
	
	\begin{proof}
	From Lemma~\ref{lemma:nonflatprime} and Corollary~\ref{cor:no-vert-faithful} we conclude that there are no non-flat regular polyhedra with a prime number of vertices.  Table~\ref{tab:flat-b} describes all such flat regular polyhedra.
	\end{proof}


\section{Polyhedra with twice a prime number of vertices }
\label{sec:twiceprime}

	We start with a useful result about vertex-faithful regular polyhedra with $2b$ vertices.  Next, we will consider the flat regular polyhedra.  The vertex-faithful regular polyhedra will be broken down by the action of the automorphism groups on the vertices.  Finally,  we conclude this section with a classification of the regular polyhedra with twice a prime number of vertices.

\begin{proposition}\label{prop:newqb}
Suppose that $\calP$ is a vertex-faithful regular polyhedron of type $\{p, q\}$ with $2b$ vertices, with $b$ a prime at least $5$,
and with $q < 2b$. Then $q = b$ if and only if $\G(\calP)$ has a normal $b$-subgroup.
Furthermore, if $\calP$ is non-flat, then $q \neq b$ and $\G(\calP)$ does not have a normal $b$-subgroup.

\end{proposition}

\begin{proof}
Suppose that $q = b$, and let $n$ be the number of Sylow $b$-subgroups. Then $|\G(\calP)| = 4b^2$, and so $n$ divides $4$
and $n \equiv 1$ (mod $b$). If $b \geq 5$, then $n = 1$ and the Sylow $b$-subgroup is normal.

Conversely, suppose that $\G(\calP)$ has a normal $b$-subgroup $S$. Since $|\G(\calP)| = 4qb$, if $S$ has order
greater than $b$, then $b$ divides $q$ and since $q < 2b$ that means $q = b$. So let us assume that $S = \langle \pi \rangle$
with $\pi$ an element of order $b$.

The orbits of $\langle \pi \rangle$ form a block system on the $2b$ vertices, and since $\calP$ is vertex-faithful,
there must be two blocks of size $b$. Let $B_1$ be the block containing the base vertex, and $B_2$ the other block.
Then we can argue similarly to the proof of Proposition~\ref{prop:no-cyclic} as follows.
Since $S$ is normal, $\rho_1$ and $\rho_2$ each either commute with $\pi$ or invert it.
Note that since the blocks have odd size, both $\rho_1$ and $\rho_2$ must fix a vertex in each block. Then if either
$\rho_i$ commutes with $\pi$, it follows that $\rho_i = \pi^{-k} \rho_i \pi^k$ for all $k$, which would demonstrate
that $\rho_i$ fixes all vertices, violating faithfulness of the action. So $\rho_1$ and $\rho_2$ both invert
$\pi$, implying that $\rho_1 \rho_2$ commutes with $\pi$. Since $\rho_1 \rho_2$ fixes one of the vertices
in $B_1$, it must fix $B_1$ pointwise (by the same argument as for $\rho_i$). The only way to maintain
a faithful action is if $\rho_1 \rho_2$ does not fix any point of $B_2$. Then if $\rho_1 \rho_2$ sends some
vertex $w$ to $w \pi^k$, it must send every $w \pi^i$ to $w \pi^{i+k}$. Thus $\rho_1 \rho_2 = \pi^{k-1}$ of
order $b$, and so $q = b$.

Thus  $q = b$ if and only if $\G(\calP)$ has a normal $b$-subgroup.

Assume now that $\calP$ is non-flat and that $q=b$.  We have just seen that this implies $\G(\calP)$ has a normal $b$-subgroup.
This normal subgroup acts on the vertices with two blocks of size $b$.   Since $\rho_0$ interchanges these blocks, this implies that every relation in $\G(\calP)$ has an even number of occurrences of $\rho_0$. In particular, $p$ is even.

The order of $\G(\calP)$ is $2qv = 4b^2$, and so $p$ is a proper even divisor of $4b^2$.
If $\calP$ is not flat, then $2 < p < 2b$, and so the only possible value of $p$ is 4.  However, the order of $\G(\calP)$ is also $2pf$ where $f$ is the number of facets of $\calP$, and thus 
$pf = 2b^2$, and so this is also impossible. 

\end{proof}

\subsection{Flat polyhedra}

	Using Proposition~\ref{prop:flat-classification} and Proposition~\ref{prop:flat-core-free}, we can classify the flat orientably regular polyhedra with $2b$ vertices (that is, of type $\{2b, q\}$) where $b$ is prime. If $b$ is an odd prime,
 	then there are only two possibilities for $\Lambda(2b,q')_{i,1}$ with $\langle \s_2 \rangle$ core-free:
	$\Lambda(2b, 2)_{-1,1}$ ($i = -1$) and $\Lambda(2b, b)_{3,1}$ ($i = 3$). The three possibilities
	for $\Lambda(p',q)_{-1,j}$ where $p'$ divides $p$ and $\langle \s_1 \rangle$ is core-free are:
	\begin{enumerate}
	\item $\Lambda(2, q)_{-1,1}$ ($j = 1$),
	\item $\Lambda(b, 2b)_{-1,-3}$. ($q = 2b$ by \cite[Prop. 4.11]{tight3}, $j = -3$), and
	\item $\Lambda(2b, q)_{-1,j}$. ($i = -1$, $j$ depends on $q$).
	\end{enumerate}
	Putting these together, we find the following possible automorphism groups of flat regular polyhedra with $2b$ vertices:
	\begin{enumerate}
	\item $\Lambda(2b, q)_{-1,1}$, with $q$ even,
	\item $\Lambda(2b, q)_{3,1}$, with $q$ divisible by $b$,
	\item $\Lambda(2b, 2b)_{-1,-3}$, and
	\item $\Lambda(2b, q)_{-1,j}$, with $q$ divisible by $2b$ but not $2b^2$, and where $j$ is the unique positive
	integer with $1 \leq j \leq q-1$ that satisfies
	$\frac{j+1}{2} \equiv -1$ (mod $b$) and $\frac{j+1}{2} \equiv 1$ (mod $q/2b$).
	\end{enumerate}
	Of these, only $\Lambda(2b, b)_{3,1}$ yields a vertex-faithful polyhedron.
	

	Now suppose that $b = 2$ (so that $\calP$ has $4$ vertices). The only possibility for $\Lambda(4,q')_{i,1}$ is $\Lambda(4,2)_{-1,1}$.
	There are two possibilities for $\Lambda(p', q)_{-1,j}$; either $\Lambda(2, q)_{-1,1}$ or
	$\Lambda(4, q)_{-1,j}$, and in the second case $q$ is divisible by 8. Thus there are two families;
	one with group $\Lambda(4,q)_{-1,1}$ with $q$ even, and one with group $\Lambda(4,q)_{-1,j}$ with $q$
	divisible by $8$ and $j$ determined by $q$.

	Now we consider the flat non-orientably regular polyhedra. By \cite[Thm. 5.10]{tight3}, the only such polyhedra
	with $2b$ vertices for an odd prime $b$ have type $\{6, 4r\}$, with $r \geq 1$ and odd. Using the dual of
	\cite[Lem. 5.9]{tight3}, we find the following groups:
	\begin{enumerate}
	\item The quotient of $[6, 4]$ by the relations $\s_2^{-1} \s_1 = \s_1^{-1} \rho_1 \s_2^2$ and $\s_2^{-1} \s_1^2 = 
	\s_1^{-2} \s_2$. In fact, we can verify with a CAS that the second relation is uneccessary.
	\item The quotient of $[6, 4r]$ by the relations 
	$\s_2^{-1} \s_1 = \s_1^2 \rho_1 \s_2^{r+1}$ and $\s_2^{-1} \s_1^2 = \s_1^{-2} \s_2^{2r-1}$,
	for $r \equiv 1$ (mod $4$).
	\item The quotient of $[6, 4r]$ by the relations 
	$\s_2^{-1} \s_1 = \s_1^2 \rho_1 \s_2^{-r+1}$ and $\s_2^{-1} \s_1^2 = \s_1^{-2} \s_2^{2r-1}$,
	for $r \equiv 3$ (mod $4$).
	\end{enumerate}
	
	In cases (b) and (c), \cite[Lem. 5.4]{tight3} shows that $\langle \s_2^4 \rangle$ is normal, and so
	for $r > 1$ the corresponding polyhedra are not vertex-faithful. We can check with a CAS that $\langle \s_2 \rangle$
	is core-free when $r = 1$ for the group in case (b).
	
	Finally, there are two families of flat non-orientably regular polyhedra with $4$ vertices (of type $\{4, q\}$), described
	in \cite[Lem. 5.9]{tight3}. The first is the quotient of $[4, 3k]$ by the relations
	$\s_2^{-1} \s_1 = \s_1^2 \rho_1 \s_2$ and $\s_2^{-2} \s_1 = \s_1^{-1} \s_2^2$, and the second is
	the quotient of $[4, 6k]$ by the relations 	$\s_2^{-1} \s_1 = \s_1^2 \rho_1 \s_2^{1+3k}$ and 
	$\s_2^{-2} \s_1 = \s_1^{-1} \s_2^2$. In fact, we can show that the second relation is unnecessary
	to define these groups. By \cite[Prop. 3.2]{tight3}, the quotient of $[4, 3k]$ by the single
	relation $\s_2^{-1} \s_1 = \s_1^2 \rho_1 \s_2$ already causes $\langle \s_2^3 \rangle$ to be normal,
	and taking the quotient by this normal subgroup yields the group of the hemicube, where the second
	relation is unnecessary. (This also demonstrates that if $k > 1$ then this polyhedron is not vertex-faithful
	since $\langle \s_2 \rangle$ has a nontrivial core.)
	A similar argument holds for the quotient of $[4, 6k]$ by the single relation
	$\s_2^{-1} \s_1 = \s_1^2 \rho_1 \s_2^{1+3k}$, where $\langle \s_2^6 \rangle$ is normal.

	We summarize all of the flat regular polyhedra with $2b$ vertices
	in Table~\ref{tab:flat-2b}.

	\begin{table}
	\begin{center} \begin{tabular}{|l|l|l|l|l|} \hline
	Type & Relations & Orientable? & Vertex-faithful? & Notes \\ \hline
		$\{4, q\}$ & $\s_2^{-1} \s_1 = \s_1^{-1} \s_2$ & Y & N & $q$ even \\ \hline
	$\{4, 2^{\alpha} k\}$ & $\s_2^{-1} \s_1 = \s_1^{-1} \s_2^j$ & Y & N & $\alpha \geq 3$, $k$ odd, $\frac{j+1}{2} \equiv 1$ (mod $k$),\\
	&&&& $\frac{j+1}{2} \equiv 2^{\alpha-2}+1$ (mod $2^{\alpha-1}$) \\ \hline
	$\{4, 3k\}$ & $\s_2^{-1} \s_1 = \s_1^2 \rho_1 \s_2$ & N & Y* & (*Vertex-faithful if $k = 1$) \\ \hline
	$\{4, 6k\}$ & $\s_2^{-1} \s_1 = \s_1^2 \rho_1 \s_2^{1+3k}$ & N & N & \\ \hline

	$\{6, 4\}$ & $\s_2^{-1} \s_1 = \s_1^{-1} \rho_1 \s_2^{2}$ & N & N & \\ \hline
	$\{6, 4r\}$ & $\s_2^{-1} \s_1 = \s_1^2 \rho_1 \s_2^{r+1}$ & N & Y* & $r \equiv 1$ (mod $4$). \\
				& $\s_2^{-1} \s_1^2 = \s_1^{-2} \s_2^{2r-1}$ & & & (*Vertex-faithful if $r = 1$) \\ \hline
	$\{6, 4r\}$ & $\s_2^{-1} \s_1 = \s_1^2 \rho_1 \s_2^{-r+1}$ & N & N & $r \equiv 3$ (mod $4$). \\ 
				& $\s_2^{-1} \s_1^2 = \s_1^{-2} \s_2^{2r-1}$ & & & \\ \hline

	$\{2b, q\}$ & $\s_2^{-1} \s_1 = \s_1^{-1} \s_2$ & Y & N & $q$ even \\ \hline
	$\{2b, q\}$ & $\s_2^{-1} \s_1 = \s_1^3 \s_2$ & Y & Y* & $q$ divisible by $b$ \\
	&&&& (*Vertex-faithful if $q = b$) \\ \hline
	$\{2b, 2b\}$ & $\s_2^{-1} \s_1 = \s_1^{-1} \s_2^{-3}$ & Y & N & \\ \hline
	$\{2b, q\}$ & $\s_2^{-1} \s_1 = \s_1^{-1} \s_2^{j}$ & Y & N & $q$ divisible by $2b$ but not $2b^2$, \\
	&&&& $\frac{j+1}{2} \equiv -1$ (mod $b$), $1$ (mod $q/2b$) \\ \hline
	\end{tabular}
	\caption{Flat regular polyhedra with 4 or $2b$ vertices, where $b$ is an odd prime}
	\label{tab:flat-2b}
	\end{center}
	\end{table}


\subsection{Non-flat polyhedra -- the primitive case}

Now we will consider non-flat vertex-faithful regular polyhedra whose automorphism groups have a primitive action on their $2b$ vertices, where $b$ is prime.

When $b \in \{2, 3, 5\}$ there do exist non-flat vertex-faithful regular polyhedra whose automorphism groups act primitively on their $2b$ vertices, for instance $\{3,3\}$, $\{5,5\}_3$, and $\{5,3\}_5$ respectively. 
On the other hand, when $b=7$, there are only two vertex-faithful regular polyhedra; both have automorphism groups which act imprimitively with blocks of 7 vertices.
When $b = 11$, we can use similar techniques as in Section~\ref{sec:smallV} to see that there are two vertex-faithful regular polyhedra with 22 vertices, both of which have automorphism groups acting imprimitively with blocks of size 11.
When $b > 11$, we will use the following result which relies on the classification of finite simple groups and the work of Liebeck and Saxl~\cite{LiebeckSaxl}.

\begin{theorem}[\cite{ChengOxley}, Theorem 1.1]
Let $b$ be a prime number.  A primitive permutation group of degree $2b$ is doubly transitive provided that $b \not = 5$.
\end{theorem}

The results above then lead us to the following theorem.

\begin{theorem}
When $b \geq 7$ is prime, there are no non-flat vertex-faithful regular polyhedra with automorphism groups that act primitively on their $2b$ vertices.
\end{theorem}
\begin{proof}
Assume to the contrary that there is a non-flat vertex-faithful regular polyhedron with automorphism group $\Gamma$ acting primitively on $2b$ vertices where $b$ is prime.  The discussion above shows that we may assume $b \geq 13$ and that $\Gamma$ acts doubly transitively.

Since $\Gamma$ acts doubly transitively on a set of size $2b$,
the stabilizer of the base vertex acts transitively on the remaining $2b-1$ vertices; in particular $| \Gamma | = (k)(2b)(2b-1)$.  Additionally $| \Gamma | = 4e$, where $e$ is the number of edges of the polyhedron.  Thus $k$ must be even.  Also $| \Gamma | = (2)(q)(2b)$, and thus $k(2b-1) = 2q$.  When $k > 2$ (and thus $k >3$), this implies that $q > 2b =v$, which contradicts Proposition~\ref{prop:qorientable}. 
Thus $k=2$, $q=2b-1$, and $| \Gamma | = 4b(2b-1)$.

Let $n$ be the number of Sylow $b$-subgroups of $\G$.  We know that $n = 1 + tb$ for some integer $t$, and that $n$ divides $8b-4$. Thus $8b-4 = r(1+tb)$ for some integer $r$. Solving this for $b$ we get $b=\frac{r+4}{8-rt}$.

Furthermore, by Proposition~\ref{prop:newqb}, we know that $t \not = 0$.
Thus, $rt \leq 7$, and in particular $b \leq r +4 \leq 11$.  

 \end{proof}


\subsection{Non-flat polyhedra -- the imprimitive case}
Now we will classify the non-flat vertex-faithful regular polyhedra with twice a prime number of vertices whose automorphism groups act imprimitively on the vertices.  We again assume that $b$ is an odd prime with $b \geq 11$.
Corollary~\ref{cor:vq} provides the unique vertex-faithful regular polyhedron of type $\{p, 2b\}$, so by Proposition~\ref{prop:qorientable}, we may assume for the remainder of this section that $q < 2b$.

\subsubsection{$b$ blocks of $2$ vertices} \label{subsub:bblocks}

	First we will show that there are no such regular polyhedra $\calP$ where the action of $\G(\calP)$ on the vertices is imprimitive with blocks of size 2. 
		Let $\G = \G(\calP)$ and let $B_1$ be the block containing the base vertex of a regular polyhedron of type $\{p,q\}$. In this case, $|\Stab(\G, B_1)| = 4q$.
		Since $q < 2b$, Proposition~\ref{prop:newqb} implies that $b$ does not divide $q$.
		Thus the Sylow $b$-subgroup of $\G$ is cyclic of order $b$, generated by an element
		$\pi$ that permutes the $b$ blocks cyclically. Note that, since the blocks have size $2$ and 
		$\rho_1$ and $\rho_2$ both fix the base vertex, they in fact both fix $B_1$ pointwise. 
		Then by Theorem~\ref{prop:group2} and Proposition~\ref{prop:newqb},
		$\G$ acts doubly transitively on the blocks. 
		In particular, $\Stab(\G, B_1)$ acts transitively on the $b-1$ remaining blocks. Then
		\[ |\Stab(\G, B_1)| = |\Stab(  \Stab(\G, B_1),B_2)| \cdot (b-1), \]
		and since $\rho_2 \in \Stab(  \Stab(\G, B_1),B_2)$, it follows that $2b-2$ divides $4q$.
		
		Let $n$ be the number of Sylow $b$-subgroups. Since $\langle \pi \rangle$ is not normal (by Proposition~\ref{prop:newqb})
		we have $n \neq 1$. Thus, we need for $4q$ to be a multiple of $2b-2$ and of $n$, where
		$n = kb+1$ for some $k \geq 1$. Since $q \leq v-1$, we have $4q \leq 8b-4$. Then $4q = (2b-2)r$ with  $1 \leq r \leq 4$.  Similarly $4q = (kb+1)s$ with $1 \leq k \leq 8$ and $1 \leq s \leq 8$. With a CAS, we find that the only solutions this system of equations
		(with odd prime $b$) occur when $b \leq 11$. We know that there are only two vertex-faithful polyhedra with $22$ vertices: the one of type $\{4, 22\}$ ($(\{4,4\}_{(11,0)})^{\pi \delta}$) and the flat one of type $\{22, 11\}$ described by the second line of Table~\ref{tab:flat-2b}.
		
	Therefore, for all primes $b \geq 5$, there is no non-flat vertex-faithful regular polyhedron with $2b$ vertices where the action of the automorphism group has blocks of size 2.
Additionally it can be checked that when $b=2$, this result also holds.  On the other hand, the octahedron has an automorphism group that acts with 3 blocks of 2 vertices.

	
\subsubsection{$2$ blocks of $b$ vertices}

Here, for all primes at least seven, we show that the non-flat vertex-faithful regular polyhedra from Corollary~\ref{cor:vq} with twice a prime number of vertices is unique.  

\begin{proposition} \label{prop:only-2xb}
For all primes $b \geq 7$, if $\calP$ is a non-flat vertex-faithful regular polyhedron with $2b$ vertices such that the action on vertices is imprimitive with two blocks of size $b$, then $q = 2b$ and $\calP$ is $(\{4,4\}_{(b,0)})^{\pi \delta}$.
\end{proposition}
	
\begin{proof}
Let $\Gamma = \G(\calP)$ and suppose that $\Gamma$ acts faithfully and imprimitively on the $2b$ vertices with two blocks each with $b$ vertices.  Let the base vertex $u$ be in the block $B_1$.  The set of fixed points of $\G_0:=\langle \rho_1, \rho_2 \rangle = \Stab(\Gamma,u)$ forms a block for $\Gamma$, so the number of fixed points of $\G_0$ is either 1, 2, $b$, or $2b$.  The number of fixed points for $\G_0$ cannot be $2b$ or $\Gamma$ would not act faithfully.  
If the number of of fixed points is 2, then there is a block system for $\Gamma$ consisting of $b$ blocks each of size 2; this was ruled out in subsection~\ref{subsub:bblocks}.  

Let $q'$ be a divisor of $q$ and let $H = \langle \s_2^{q'} \rangle$. Recall from the proof of Proposition~\ref{prop:steves-thm} that if $q' < q/2$, then $ \langle \s_2^{q'} \rangle$ fixes a vertex $u \alpha$ if and only if $\alpha$ normalizes $H$. 

Suppose that $H$ fixes a vertex $u \alpha$ in $B_2$. Then $H \alpha = \alpha H$, which implies that $H$ fixes a vertex $w$ if and only if it fixes $w \alpha$. Since $u$ is in $B_1$ and $u \alpha$ is in $B_2$, that means that $\alpha$ interchanges $B_1$ with $B_2$, and so the number of vertices fixed by $H$ must be even. By Proposition~\ref{prop:steves-thm} and the fact that $\calP$ is vertex-faithful, this implies that either $H$ fixes exactly two vertices, or that $q' = q/2$.

Consider what this implies when $q'=1$ and  $H = \langle \s_2 \rangle$.
If $\s_2$ fixes a vertex $v$, then either both $\rho_1$ and $\rho_2$ fix $v$, or there is a distinct vertex $w = (v) \rho_1 = (v) \rho_2$.  
If the number of fixed points of $H$ is 2, then since $\G_0$ fixes the base vertex $u$,  then $\G_0$ also fixes the other fixed point of $H$.  However, we have already ruled out the case where $\G_0$ has 2 fixed points, so the number of fixed points of $H$ cannot be 2, and thus $\s_2$ cannot fix any vertex in $B_2$

Therefore, the orbits of vertices in $B_2$ under $ \langle \s_2 \rangle$ all have size at least 2. Suppose that there are two orbits of coprime size $2 \leq m < n$, where $m$ and $n$ are divisors of $q$. Note that $m \neq q/2$ since if $m = q/2$, then $n > m$ would force $n = q$, in which case $m$ and $n$ are not coprime. The above argument shows that $\langle \s_2^m \rangle$
fixes exactly two vertices.  However, in addition to fixing the base vertex u, the group $\langle \s_2^m \rangle$ fixes at least $m$ vertices of $B_2$ that are in the same $ \langle \s_2 \rangle$ orbit. Therefore, the orbit sizes in $B_2$ all have a common factor. Since there are no orbits of size 1 and the block has prime size $b$, it follows that the orbit must be the entire block $B_2$. Therefore, $b$ divides $q$, and by Proposition~\ref{prop:newqb}, $q \neq b$. Then Proposition~\ref{prop:qorientable} and Corollary~\ref{cor:vq} imply the rest of the claim.
\end{proof}

\subsubsection{Representation of a flat polyhedron}
Let us also examine what happens to the above argument when $q=b$ and thus $\calP$ is flat.  
Consider the vertex CPR graph of $\G$. Since $\s_2$ has prime order $b$ and it fixes one vertex of $B_1$, all the orbits in $B_1$ must have size $1$, which is to say that $\s_2$ fixes $B_1$ pointwise.  Thus, the $b-1$ vertices in $B_1$ other than the base vertex are connected in pairs by parallel edges labeled 1 and 2. This completely determines the induced subgraph on vertices in $B_1$.   
By the argument of Proposition~\ref{prop:only-2xb}, the vertices of $B_2$ form a single $\langle \s_2 \rangle$ orbit, and so 
the induced subgraph on vertices in $B_2$ is an alternating path of length $b$.  
 This leaves only 1 possible vertex CPR graph for $\Gamma$:

$$\includegraphics[height=6cm]{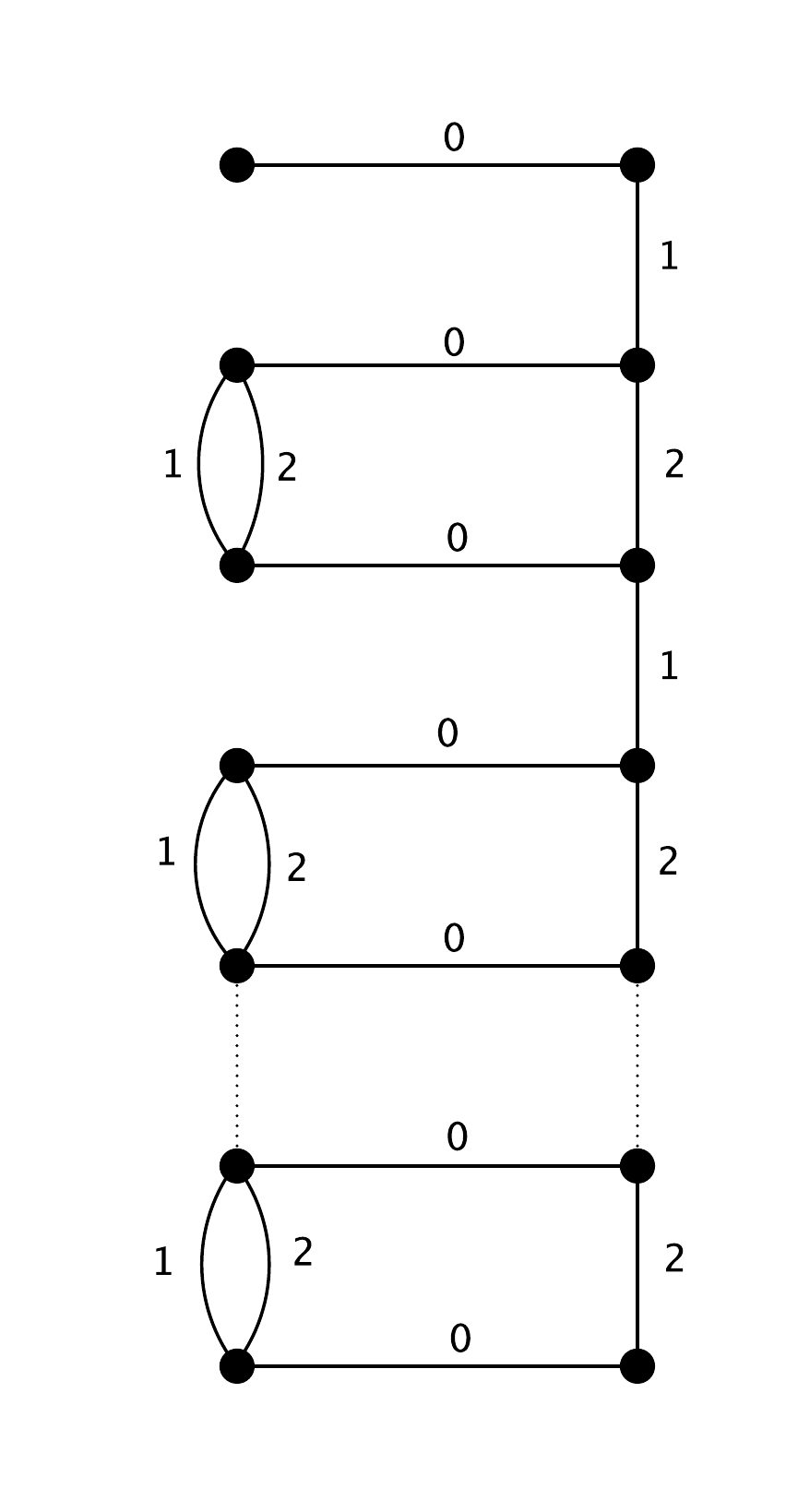}$$

In this case $p$ is equal to $2b$ and $\Gamma$ is the automorphism group of the flat vertex-faithful regular polytope of type $\{2b, b\}$ (see Table~\ref{tab:flat-2b}).

\subsection{Classification of all regular polyhedra with twice a prime number of vertices}

We have seen the classification of the regular flat polyhedra with $2b$ vertices, and we have constructed the unique non-flat vertex-faithful regular polyhedron with $2b$ vertices for each prime $b \geq 7$.   Now we will classify the non-flat regular polyhedra with $4$, $6$, and $10$ vertices. Any such polyhedron covers
one of the vertex-faithful polyhedra in Table~\ref{tab:small1}.

\begin{lemma}
Suppose that $\calP$ is a non-flat regular polyhedron of type $\{p, q\}$ that is not vertex-faithful, 
and let $\calQ$ be the vertex-faithful regular polyhedron of type $\{p, q'\}$ such that $\calP$ covers $\calQ$.
\begin{enumerate}
\item If $\calQ = \{3, 3\}$, then $q = 6$ or $12$.
\item If $\calQ = \{3, 4\}$, then $q = 8, 12$, or $24$.
\item If $\calQ$ is one of the following, then $q = 2q'$: $\{3, 5\}*60, \{5,5\}*60, \{4,6\}*120, \{5,3\}*60, \{5,6\}*120a,  
\{6,6\}*120$.
\end{enumerate}
\end{lemma}

\begin{proof}
For parts (a) and (b), Proposition~\ref{prop:vf-quos}(d) says that $\rho_0$ inverts $\s_2^{q'}$ in $\G(\calP)$.
Adding the relation $\rho_0 \s_2^{q'} \rho_0 = \s_2^{-q'}$ to the group $[3, \infty]$ with $q' = 3$ causes
a collapse to a quotient of $[3, 12]$, and when $q' = 4$ we get a collapse to a quotient of $[3, 24]$.

For part (c), we can appeal to Corollary~\ref{cor:odd-zigzags}, using the information in Table~\ref{tab:small1}.
\end{proof}

Given this bound on $q$, and thus on the size of the group, we can use~\cite{atlas} to finish the classification of non-vertex-faithful regular polyhedra.  Table~\ref{tab:small-non-vf} summarizes such polyhedra with $4$, $6$,
or $10$ vertices.

\begin{table}[]
\begin{center} \begin{tabular}{|l|l|l|l|l|l|l|} \hline
$v$ & Atlas Canonical Name & Also Known As &   $|z_1|$ & $|h|$ & $ | z_2 |$ & Universal\\ \hline
4&	  $\{3,6\}*48$  & $\{3,6\}_{(2,0)}$, $\{3,6\}_4$ & 4 & 6 & 4 & Y \\
4&	  $\{3,12\}*96$ & & 8 & 12 & 8 & N \\
\hline
6&	  $\{3,8\}*96$ & & 12 & 8 & 8 & N \\
6&	  $\{3,12\}*144$ & & 6 & 12 & 4 & Y \\
6&	  $\{3,24\}*288$ & & 12 & 24 & 8 & N \\
6&	  $\{3,10\}*120a$ & $\{3,10\}_5$ & 5 & 10 & 6 & Y \\
6&	  $\{3,10\}*120b$ & & 10 & 10 & 3 & Y \\
6&	  $\{5,10\}*120a$ & $\{5,10\}_3$ & 3 & 6 & 10 & Y \\
6&	  $\{5,10\}*120b$ & & 6 & 6 & 5 & N \\
\hline
10&   $\{5,6\}*120b$ & & 5 & 10 & 10 & N \\
10&	  $\{5,6\}*120c$ & & 10 & 10 & 5 & N \\
\hline

\end{tabular}
\caption{The non-flat, non-vertex-faithful regular polyhedra with 4, 6, or 10 vertices.}
\label{tab:small-non-vf}
\end{center}
\end{table}

With the small cases fully understood, this leads us to the following theorem.

\begin{theorem} \label{thm:2b-classification}
A regular polyhedron with $2b$ vertices (where $b$ is prime) is one of the following: 
\begin{itemize}
\item A vertex-faithful polyhedron with $b \leq 5$, described in Table~\ref{tab:small1},
\item A flat polyhedron described in Table~\ref{tab:flat-2b}, 
\item A non-vertex-faithful polyhedron with $b \leq 5$, described in Table~\ref{tab:small-non-vf}, or 
\item $(\{4,4\}_{(b,0)})^{\pi \delta}$, which is the unique non-flat vertex-faithful regular polyhedron of type $\{4,2b \}$ with $b \geq 3$. 
\end{itemize}
\end{theorem}

\begin{proof}
The analysis in this section proves that the only vertex-faithful non-flat regular polyhedron with $2b$ vertices and $b \geq 7$ is $(\{4,4\}_{(b,0)})^{\pi \delta}$. Lemma~\ref{lem:nocover} then says that no other (non-vertex-faithful) polyhedra with $2b$ vertices cover this polyhedron.
\end{proof}

\section{Polyhedra with a prime squared number of vertices}
\label{sec:primesq}

\subsection{Flat polyhedra}

	The flat regular polyhedra with $4$ vertices were covered in Section~\ref{sec:twiceprime} and the corresponding polyhedra are in Table~\ref{tab:flat-2b}. So we will assume that $b$ is an odd prime.
	Again, we use Proposition~\ref{prop:flat-classification} and Proposition~\ref{prop:flat-core-free} to determine
	the flat orientably regular polyhedra of type $\{b^2, q\}$.
	The only possibility for $\Lambda(b^2,q')_{i,1}$ with $\langle \s_2 \rangle$ core-free is $\Lambda(b^2, 2)_{-1,1}$.
	There are two possibilities for $\Lambda(p', q)_{-1,j}$ with $p'$ dividing $b^2$ and $\langle \s_1 \rangle$
	core-free: either $\Lambda(b^2, 2b^2)_{-1,-3}$ or $\Lambda(b, 2b)_{-1,-3}$. Putting these together,
	we find two flat orientably regular polyhedra with $b^2$ vertices:
	one of type $\{b^2, 2b^2\}$ with group $\Lambda(b^2, 2b^2)_{-1,-3}$, and one of type $\{b^2, 2b\}$ with
	group $\Lambda(b^2, 2b)_{-1,-3}$.
	
	Moving on to non-orientably regular flat polyhedra, \cite[Thm. 5.10]{tight3} implies that there are no non-orientably regular 
	flat polyhedra of type $\{b^2, q\}$ for any odd prime $b \geq 5$. When $b = 3$, there is a single non-orientably regular
	flat polyhedron, which has type $\{9, 4\}$. The group of this polyhedron is the quotient of $[9, 4]$ by the extra
	relation $(\rho_0 \rho_1 \rho_2 \rho_1)^2 \rho_2 = 1$ \cite{atlas}. 

	We summarize the flat regular polyhedra with $b^2$ vertices in Table~\ref{tab:flat-bsquared}.

	\begin{table}
	\begin{center} \begin{tabular}{|l|l|l|l|l|} \hline
	Type & Relations & Orientable & Vertex-faithful & Notes \\ \hline
	$\{9, 4\}$ & $(\rho_0 \rho_1 \rho_2 \rho_1)^2 \rho_2 = 1$ & N & N & \\ \hline
	$\{b^2, 2\}$ & $\s_2^{-1} \s_1 = \s_1^{-1} \s_2$ & Y & N & \\ \hline
	$\{b^2, 2b\}$ & $\s_2^{-1} \s_1 = \s_1^{-1} \s_2$ & Y & N & \\ \hline
	$\{b^2, 2b^2\}$ & $\s_2^{-1} \s_1 = \s_1^{-1} \s_2$ & Y & N & \\ \hline
	\end{tabular}
	\caption{Flat regular polyhedra with $b^2$ vertices, where $b$ is an odd prime}
	\label{tab:flat-bsquared}
	\end{center}
	\end{table}
	
\subsection{Non-flat polyhedra}
	
Fully classifying the non-flat regular polyhedra with $b^2$ vertices appears to be somewhat more difficult than the problems we have considered so far. Instead, we will consider only the smallest non-flat regular polyhedra with $b^2$ vertices (where by ``smallest'' we mean having the fewest flags). Since the number
of flags is $2qv = 2qb^2$, finding the smallest polyhedra amounts to finding the smallest value of $q$. Furthermore, since
the number of flags is also $4e$, we find that $q$ must be even if $b$ is odd. If $\calP$ is not flat, then $q \neq 2$. Then the smallest 
possible value for $q$ is $q = 4$, and there is at least one non-flat regular polyhedron with $b^2$ vertices: the toroidal map
$\{4, 4\}_{(b,0)}$. Let us classify the non-flat regular polyhedra of type $\{p, 4\}$ with $b^2$ vertices.
	
\begin{lemma}\label{lem:normalS}
Suppose that $b$ is an odd prime and that $\calP$ is a non-flat regular polyhedron of type $\{p, 4\}$ with $b^2$ vertices. Then:
\begin{enumerate}
\item $\G(\calP)$ has a normal Sylow $b$-subgroup.
\item $p \neq b$.
\end{enumerate}
\end{lemma}

\begin{proof}
First, let $n$ be the number of Sylow $b$-subgroups of $\G (\calP)$. We have $|\G(\calP)| = 2qv = 8b^2$.
Then by the Sylow theorems, $n$ divides 8 and $n \equiv 1 ( \textrm{mod } b)$.  
Clearly if $b \ge 11$ then $n = 1$. For $3 \le b \le 7$, we can verify the claim using \cite{atlas} and a CAS.

Now, suppose that $p = b$ and let $S$ be the normal Sylow $b$-subgroup. Since $|\G(\calP)| = 8b^2 = 2pf$, the polyhedron
$\calP$ has $4b$ faces. Consider the action of $\G(\calP)$ on the faces. Since $S$ is normal, the orbits of the faces
under $S$ form a system of blocks for $\G$. Let $B_1$ be the block containing the base face and let $B_2 = (B_1) \rho_2$.
Then $\rho_0$ stabilizes $B_1$ and thus $B_2$ (since $\rho_0$ commutes with $\rho_2$), and $\rho_1$ stabilizes
$B_1$. Furthermore, since $\rho_0 \rho_1$ has order $b$, it must lie in $S$, and so it stabilizes $B_2$. Thus $\rho_1$
also stabilizes $B_2$, and so the block system consists of these two blocks only. The size of each block divides
$b^2$, and so the total number of faces must divide $2b^2$, contradicting that there are $4b$ faces.
\end{proof}

We end with a classification of the smallest regular non-flat polyhedra with a prime squared number of vertices.

\begin{theorem}
For $b \geq 3$, up to isomorphism, there are exactly two smallest regular non-flat polyhedra with $b^2$ vertices: 
the toroidal map $\{4, 4\}_{(b,0)}$ and its Petrial of type $\{2b, 4\}$. \end{theorem}

\begin{proof}
Suppose that $\calP$ is a smallest regular non-flat polyhedron with $b^2$ vertices. We have already established that $\calP$ must have type $\{p, 4\}$ for some $p$. Furthermore, the toroidal map $\{4,4\}_{(b,0)}$ and its Petrial of type $\{2b, 4\}$ both have $b^2$ vertices. First, let us show that there is no other possible value of $p$. We have that $|\G(\calP)| = 8b^2 = 2pf$. In order for $\calP$ to be non-flat, we need $2 < p < b^2$ and for $p$ to properly divide $4b^2$. Lemma~\ref{lem:normalS} rules out the case $p = b$, and if $p = 4b$, then $f = b$, and the dual of Theorem~\ref{thm:prime-verts} implies that $\calP$ must be flat. So $p = 4$ or $p = 2b$.

Now, if $p = 4$, then $\calP$ is a regular polyhedron of type $\{4, 4\}$, and so it must be either the toroidal map $\{4, 4\}_{(a,0)}$, which has $a^2$ vertices, or the toroidal map $\{4,4\}_{(a,a)}$, which has $2a^2$ vertices. So $\calP$ must be $\{4,4\}_{(b,0)}$. If instead $p = 2b$, then $|\G(\calP)| = 8b^2 = 2pf$ implies that $f = 2b$. Then $\calP^{\delta}$ is a non-flat regular polyhedron of type $\{4, 2b\}$ with $2b$ vertices. By Theorem~\ref{thm:2b-classification}, $\calP^{\delta}$ must be $(\{4,4\}_{(b,0)})^{\pi \delta}$, proving the claim.
\end{proof}


\textbf{Acknowledgement:} We would like to thank the referees for many useful comments.

\bibliographystyle{amsplain}
\bibliography{gabe}

\end{document}